\documentclass{article}

\usepackage{xcolor}

\definecolor{linkblue}{RGB}{0,90,180}
\definecolor{lightlinkblue}{RGB}{0,90,180}

\usepackage[
    pdftex,
    colorlinks=true,
    bookmarksopen=true,
    linkcolor=lightlinkblue,
    citecolor=linkblue,      
    urlcolor=linkblue       
]{hyperref}
\usepackage{xspace}
\usepackage[
    a4paper,
    left=3cm,
    right=3cm,
    top=3cm,
    bottom=3cm
]{geometry}

\usepackage{booktabs}
\usepackage[T1]{fontenc}
\usepackage[utf8]{inputenc}
\usepackage{float}
\usepackage{bm}
\usepackage{libertine}
\usepackage{makecell}
\usepackage{bbm}
\usepackage{enumitem}
\usepackage{mathrsfs}
\usepackage{amssymb}
\usepackage{amsmath}
\usepackage{upgreek}
\usepackage{mathtools}
\usepackage{multirow,multicol}
\usepackage{amsthm}
\usepackage{comment}
\usepackage[font=small,labelfont=bf]{caption}
\usepackage{cleveref}
\usepackage{algorithm}
\usepackage{algpseudocode}

\theoremstyle{plain}
\newtheorem{theorem}{Theorem}[section]

\newtheorem{corollary}[theorem]{Corollary}

\newtheorem{remark}[theorem]{Remark}

\usepackage{graphicx} 
\usepackage{quoting}
\usepackage{subcaption}
\usepackage[numbers,sort&compress]{natbib}

    \definecolor{blue1}{RGB}{28, 69, 135}
    \definecolor{green1}{RGB}{58, 121, 30}
    \definecolor{red1}{RGB}{164, 2, 0}
    
    \newcommand{\pt}{\partial_t}
    \newcommand{\px}{\partial_x}
    \newcommand{\pxx}{\partial_{xx}}
    \newcommand{\bomega}{\boldsymbol{\omega}}
    
    \newcommand{\y}{{\boldsymbol{y}}}
    \newcommand{\f}{{\boldsymbol{f}}}
    \newcommand{\Y}{\boldsymbol{Y}}
    \newcommand{\F}{\boldsymbol{F}}

    \def\R{\mathbb R}
    
    \def\to{\rightarrow}
    \def\E{\mathbb E}

\usepackage{lineno}
\usepackage{pifont}
\newif\ifredline
\ifdefined\cleanversion\redlinefalse\else\redlinetrue\fi
\ifredline
  
  \newenvironment{revblock}{\begingroup\color{red}}{\endgroup}
\else

\fi

\title{Random Batch Methods for Discretized PDEs on Graphs}

\hypersetup{
  pdftitle={Random Batch Methods for Discretized PDEs on Graphs},
  pdfauthor={Martin Hernandez and Enrique Zuazua},
  pdfkeywords={Random Batch Methods, Numerical Methods, Partial Differential Equations on Graphs, Heat Equation, Optimal Control.}
}

\author{
Martin Hernandez\textsuperscript{1,*} \qquad
Enrique Zuazua\textsuperscript{2,3,4}
}

\date{\vspace{-1cm}}

\begin{document}

\maketitle

\footnotetext[1]{
Department of Statistics and Data Science, University of California, Los Angeles, CA, USA
}

\footnotetext[2]{
Chair for Dynamics, Control, Machine Learning, and Numerics
(Alexander von Humboldt Professorship), Department of Mathematics,
Friedrich-Alexander-Universität Erlangen-Nürnberg,
91058 Erlangen, Germany.
}

\footnotetext[3]{
Departamento de Matemáticas, Universidad Autónoma de Madrid,
28049 Madrid, Spain.
}

\footnotetext[4]{
Chair of Computational Mathematics, Universidad de Deusto,
Av. de las Universidades 24, 48007 Bilbao, Basque Country, Spain.
}

\renewcommand{\thefootnote}{\fnsymbol{footnote}}
\footnotetext[1]{Corresponding author. Email: \texttt{martinh@ucla.edu}}
\renewcommand{\thefootnote}{\arabic{footnote}}

\begin{abstract}
    {
    Diffusion and transport processes on networks arise in applications such as pipeline systems, traffic and information flow, and energy and thermal networks. In these settings the dynamics are often modeled by partial differential equations (PDEs) on graph domains, and their simulation becomes increasingly expensive as the network size grows. The Random Batch Method (RBM) offers a way to reduce this cost, while preserving the large-scale behavior in expectation.}
    
    {In this paper, we study a discretize+RBM strategy for the heat equation. After a finite element discretization in space, we combine it with RBM by decomposing the resulting discrete operator into a finite sum of submatrices and, on short time intervals, replacing the full discrete operator by a randomly chosen partial sum. This defines a computationally cheaper, unbiased RBM approximation of the semi-discrete dynamics. For the one-dimensional heat equation, we prove mean-square convergence of this RBM approximation to the exact PDE solution under a suitable coupling between the short-interval length and the spatial discretization step, and we extend this analysis to a linear–quadratic optimal control problem. The same framework is then applied to the heat equation on metric graphs, and we validate it numerically, showing that the discretize+RBM scheme attains accuracy comparable to a full-graph solver while reducing computational resources. The method can be interpreted as a randomized domain-decomposition strategy on graphs. We emphasize that, beyond the heat equation, the underlying discretize+RBM construction can in principle be applied to a wider class of linear PDEs.}
    \end{abstract}
\setcounter{tocdepth}{1}
\tableofcontents
\vspace{0.5cm}

\section{Introduction}

    \subsection{Motivation}
    {The simulation and control of transport and diffusion processes on large networks are central tasks in many areas of applied science and engineering. Typical examples include gas transport in high-pressure pipeline systems, thermal propagation in district heating networks, simplified models of power grids, traffic or information flow on networked infrastructures, and diffusion phenomena on abstract graphs \cite{MR3728372,Leuge,MR3243602}. In all these cases, the underlying dynamics can often be described by PDEs posed on a graph: each edge represents a one-dimensional conduit (pipe, line, road segment, etc.), while the vertices encode junction conditions and conservation laws. The resulting models are high-dimensional and strongly coupled, and their numerical treatment quickly becomes expensive in memory and runtime as the network is refined or extended \cite{MR3815942, MR4175147}. Developing robust algorithms that exploit the network structure and enable faster simulations with reduced resource consumption is therefore a key objective.}
    
    {A notable recent contribution in this direction is the Random Batch Method (RBM), originally introduced to efficiently simulate large-scale interacting particle systems \cite{ShietalRBM}. RBM achieves efficiency by randomly grouping interactions into small batches and evolving only these over short time subintervals (see also \cite{DAUS2022111220} for multicomponent systems). The underlying idea is to replace a large deterministic coupled system by a sequence of smaller randomized subsystems whose solution approximates the full dynamics in expectation. This perspective is closely related to domain decomposition and to mini-batch strategies in machine learning, such as Mini-Batch Gradient Descent \cite[Section 3]{MR3797719}, where only subsets of data or parameters are processed at each iteration in order to reduce computational cost.}
    
    RBM is also adaptable to high-dimensional linear dynamical systems, where it involves decomposing the governing matrix into smaller sub-matrices. The time interval is partitioned into subintervals, during each of them a randomly selected batch of sub-matrices defines a randomized dynamical system. As shown in \cite{MR4433122}, the solution of this random system converges in expectation to the solution of the original system when the subintervals are sufficiently small. This methodology has been proven effective for optimal control problems as well \cite{MR4433122}.
    
    \subsection{Main contribution}
    {In this paper we develop a discretize+RBM strategy for linear parabolic PDEs and their associated linear–quadratic optimal control problems. The PDE is first discretized in space (here by the finite element method), and the resulting stiffness matrix is decomposed into a finite sum of submatrices. RBM is then used to construct a random, time-dependent approximation by activating only a subset of these submatrices on short time intervals.}
    
    {The analysis focuses on the one-dimensional heat equation and on heat equations on metric graphs. For the heat equation on an interval, we prove mean-square convergence of the RBM approximation to the exact PDE solution under a suitable coupling between the FEM mesh size and the RBM switching parameter, and we extend this result to a linear–quadratic optimal control problem, obtaining estimates for both the random state and the random control. The same discretize+RBM framework is then transferred to heat equations on metric graphs: after assembling the graph-based stiffness matrix from edge contributions, we apply RBM to a graph-based matrix decomposition and validate the approach numerically. The experiments show that the method attains accuracy comparable to a full-graph implicit Euler solver while reducing computational resources; from the discrete viewpoint, it can be interpreted as a randomized domain-decomposition strategy on graphs. Although the detailed analysis is carried out for the heat equation and a specific FEM discretization, the underlying discretize+RBM construction can in principle be applied to a broader class of linear PDEs and spatial discretizations.}
    
    {To keep the exposition transparent, we first study the one-dimensional heat equation on a single interval, which plays the role of a prototype edge. For this problem we combine the FEM semi-discretization with RBM and derive error estimates. We then pass to heat equations and optimal control on metric graphs, where the global stiffness matrix is obtained by assembling the edge contributions, and apply the same discretize+RBM strategy at this matrix level.}

    \subsection{Related Work}
    The RBM was initially proposed for finite-dimensional  interacting particle systems, leading to the development of convergence guarantees \cite{MR4230431} and tools for managing antisymmetric interactions \cite{MR4300142}. The approach has been expanded to linear dynamical systems and integrating it with Model Predictive Control (MPC), providing convergence estimates \cite{MR4307003, veldman2023stability}. The stochastic process model introduced in \cite{Latz} represents another adaptation of the RBM concept, demonstrating convergence as the step size diminishes.
    
    The application of RBM to PDEs includes analyses of particle-interacting systems' mean-field limits \cite{MR4361973} and on the homogeneous Landau and Fokker-Planck equations \cite{MR4436794, MR4744252}. These studies employ RBM to approximate PDEs through interacting particle systems. Our approach, however, directly utilizes classical PDE discretization methods such as FEM. Other methodologies that combine discretization with randomization include the work of \cite{eisenmann2022randomized} on parabolic PDEs and \cite{corella2024minibatchdescentsemiflows} on gradient flows, both offering insights into convergence behaviors under various conditions.
    
    \subsection{Outline}
    The remainder of the article is structured as follows: \Cref{sec:prelim} introduces our model and outlines the FEM for the one-dimensional heat equation, along with a detailed description of the RBM, as discussed in \cite{MR4433122}. \Cref{sec:main_results} is devoted to presenting our main findings on convergence for PDE optimal control. \Cref{sec:numerical} presents numerical simulations to validate our results and illustrate the efficacy of the random graph decomposition approach.

    \section{Preliminaries}\label{sec:prelim}
    Let us first consider $L,\,T > 0$ and the following one-dimensional heat equation:
    \begin{align}\label{eq:heat_eq}
        \begin{cases}
            \pt y - \pxx y = f, \quad &(x,t) \in (0,L) \times (0,T),\\
            y(0,t) = y(L,t) = 0, & t \in (0,T),\\
            y(x,0) = y^0(x),  &x \in (0,L).
        \end{cases}
    \end{align}
    Here, $y^0\in L^2(0,L)$ is the initial condition, and $f\in L^2(0,T;H^{-1}(0,L))$ is a source term. According to \cite[Section 7.1.2]{MR1625845}, problem \eqref{eq:heat_eq} admits a unique weak solution $y\in C([0,T];L^2(0,L))\cap L^2(0,T;H_0^1(0,L))$ with $\partial_t y\in L^2(0,T;H^{-1}(0,L))$.
    \smallbreak
    \subsection{The Finite Element Method}\label{section_FEM}
    In this section, we introduce FEM to semi-discretize \eqref{eq:heat_eq}.
    
    Consider a spatial mesh step $h = L/(N + 1)$ with $N \in \mathbb{N}$, so that the partition defined by the nodes 
    \begin{align}\label{eq:spatial_nodes}
         x_j = jh,\quad j \in \{0, \dots, N + 1\},
    \end{align}
    decomposes the interval $[0, L]$ into $N + 1$ subintervals of length $h$. Let us introduce the continuous piecewise linear functions $\phi_j(x)$, such that $\phi_j(x_l) = \delta_{jl}$ for $j \in \{ 1,\dots, N \}$ and $l \in \{ 0,\dots, N + 1 \}$, where $\delta$ denotes the Kronecker delta. We define the finite-element space $V_{h}$ as the span of $\{\phi_j\}_{j=1}^N$.
    
    Denote by $y_h^0=P_hy^0$ the $L^2(0,L)$-orthogonal projection of $y^0$ onto $V_h$. Consider the semi-discrete variational formulation of \eqref{eq:heat_eq} as follows: Find $y_h\in H^1(0,T;V_h)$ such that
    \begin{align}\label{eq:variational_fem}
        \begin{cases}
       \displaystyle \int_0^L \pt y_h(x,t) \phi_j(x)\,dx + \int_0^L \px y_h(x,t) \px \phi_j(x)\,dx  = \int_0^L f(x,t)\phi_j(x)\,dx,\vspace{0.1cm}\\
       \displaystyle \int_0^L y_h(x,0)\phi_j(x)\,dx = \int_0^L y^{0}(x)\phi_j(x)\,dx,
        \end{cases}
    \end{align}
   for almost every $t\in(0,T)$ and every $j\in\{1,\dots,N\}$. The following estimate follows by combining the smooth-initial-data estimate in \cite[Chapter 3, Theorem 3.1]{Thomee2006},  and the inhomogeneous estimate in \cite[Theorem 5.1]{FujitaMizutani1976}.

\begin{theorem}\label{th:fin_elem1}
Let $0<\theta<1$, $y^0\in H^2(0,L)\cap H_0^1(0,L)$ and $f\in C^{0,\theta}([0,T];L^2(0,L))$. Let $y$ be the solution of \eqref{eq:heat_eq} and $y_h$ the solution of \eqref{eq:variational_fem} given by \eqref{eq:def_yh_fem}. Then there exists a constant $C>0$, independent of $h$, such that
\begin{align}\label{eq:estimation_finite_elements}
\sup_{0\leq t\leq T}\|y(t)-y_h(t)\|_{L^2(0,L)}^2\leq Ch^4\left(\|y^0\|_{H^2(0,L)}^2+\|f\|_{C^{0,\theta}([0,T];L^2(0,L))}^2\right).
\end{align}
\end{theorem}
    
    Observe that \eqref{eq:variational_fem} defines a system of ordinary differential equations. Indeed, since $y_h\in H^1(0,T;V_h)$, the semi-discrete solution can be written as
    \begin{align}\label{eq:def_yh_fem}
        y_h(x,t)=\sum_{j=1}^N y_j(t)\phi_j(x),
    \end{align}
    with coefficients $y_j\in H^1(0,T)$ and the corresponding finite-dimensional state $\y_h(t) = (y_1(t), \dots, y_N(t))$ solves
    \begin{align}\label{eq:fem_system}
\begin{cases}
E_h\partial_t\y_h+R_h\y_h=\mathbf F_h, & t\in(0,T),\\
E_h\y_h(0)=\y_h^0.
\end{cases}
\end{align}
    where the mass and stiffness matrices $E_h,\,R_h \in \R^{N \times N}$ are given by
    \[
    E_h(j,k)=\int_{0}^L\phi_j(x)\phi_k(x)\,dx, \quad R_h(j,k)=\int_0^L \px \phi_j(x) \px\phi_k(x)\,dx,
    \]
    for $j,k\in\{1,\dots,N\}$, and the load vectors $\y_h^0$ and $\mathbf F_h(t)$ are defined component-wise by
\begin{align*}
(\y_h^0)_j=\int_0^L y^0(x)\phi_j(x)\,dx,\qquad (\mathbf F_h(t))_j=\int_0^L f(x,t)\phi_j(x)\,dx,
\end{align*}
respectively. In the present setting, these matrices can be computed explicitly and take the form
    \begin{align}\label{eq:stiffness_mass_matrices}
        E_h = h\begin{pmatrix}
            2/3 & 1/6 & 0 & \cdots & 0 \\
            1/6 & \ddots & \ddots & \ddots & \vdots \\
            0 & \ddots & \ddots & \ddots & 0 \\
            \vdots & \ddots & \ddots & \ddots & 1/6 \\
            0 & \cdots & 0 & 1/6 & 2/3
        \end{pmatrix}, \quad
        R_h = \frac{1}{h} \begin{pmatrix}
            2 & -1 & 0 & \cdots & 0 \\
            -1 & \ddots & \ddots & \ddots & \vdots \\
            0 & \ddots & \ddots & \ddots & 0 \\
            \vdots & \ddots & \ddots & \ddots & -1 \\
            0 & \cdots & 0 & -1 & 2
        \end{pmatrix}.
    \end{align}
    
    Of course, this finite-dimensional interpretation of the FEM approximation of the heat equation can be extended to the multi-dimensional setting in arbitrary domains, although, in that case, the matrix representation is less explicit.
    
    \subsection{Overview of the RBM}\label{sec:rbm}
    {
    In this section, we provide an introduction to the RBM, following \cite{MR4433122}, in the context of the semi-discrete system \eqref{eq:fem_system}.}
    
    {Recall that the finite element semi-discretization of \eqref{eq:heat_eq} leads to the finite-dimensional dynamics
    \eqref{eq:fem_system}, governed by the mass and stiffness matrices $E_h$ and $R_h$ defined in \eqref{eq:stiffness_mass_matrices}. The RBM will provide a random surrogate of \eqref{eq:fem_system} in which, over short time intervals, only a part of the stiffness matrix $R_h$ is active, while the full dynamics are recovered in expectation.}
    
    The construction starts from a fixed decomposition of the stiffness matrix into $M$ submatrices:
    \begin{align}\label{eq:decom_A}
        R_h = \sum_{m=1}^M R_{h,m},
    \end{align}
    for some matrices $R_{h,m}\in\mathbb R^{N\times N}$, $m\in\{1,\dots,M\}$, satisfying $\mathbf v^\top R_{h,m}\mathbf v\geq0$ for every $\mathbf v\in\mathbb R^N.$ This condition is automatically satisfied when the matrices $R_{h,m}$ are obtained by grouping symmetric positive-semidefinite local finite-element stiffness contributions. This construction will also be natural for the graph discretization considered in Section~\ref{sec:numerical}.
    
    To randomize the dynamics in time, we first introduce a partition of the time interval. Let $\delta=T/K>0$ be the RBM switching parameter, where $K\in\mathbb N$, and define the time intervals $[(k-1)\delta,k\delta)$ for $k\in\{1,\dots,K\}$. On each such interval, we shall activate only a subset of the matrices $R_{h,m}$.
    
    {The selection of these subsets is described as follows. Let $M \in \mathbb{N}$ be fixed and introduce a family of identically and independently distributed random variables $\{\omega_k\}_{k=1}^K$, each taking values in $\{1, \dots, 2^M\}$. For each $k \in \{1, \dots, K\}$, denote by $p_i \in [0,1]$ the probability that $\omega_k$ equals $i \in \{1, \dots, 2^M\}$, and assume that
    \begin{align*}
        \sum_{i=1}^{2^M} p_i = 1.
    \end{align*}
    We denote by
    \begin{align*}
        \boldsymbol{\omega} = (\omega_1,\dots,\omega_K) \in \{1,\dots,2^M\}^K
    \end{align*}
    the full sequence of random choices.}
    
    {Let $\mathcal{P}(\{1,\dots,M\})$ be the power set of $\{1,\dots,M\}$, and enumerate its elements as $\{S_i\}_{i=1}^{2^M}$. For each $k \in \{1,\dots,K\}$, the value of $\omega_k$ determines a random subset
    \[
        S_{\omega_k} \in \mathcal{P}(\{1,\dots,M\}),
    \]
    so that the probability law of $S_{\omega_k}$ is induced by that of $\omega_k$. This construction of the random subsets $S_{\omega_k}$ allows us to form partial sums of the full matrix $R_h$ and thereby define a random approximation as follows.}
    
    {For $t \in [(k-1)\delta,k\delta)$ and $k \in \{1,\dots,K\}$, we set
    \begin{align}\label{eq:definition_matrix_R_rb}
         R_{rb}(\boldsymbol{\omega}, t)
         =
         \sum_{m \in S_{\omega_k}} \frac{R_{h,m}}{\pi_m},
    \end{align}
    where the normalization constants $\pi_m > 0$ are given by
    \begin{align*}
        \pi_m:=\sum_{i:\,m\in S_i}p_i>0,\quad m\in\{1,\dots,M\}.
    \end{align*}
    By construction, the matrix $R_{rb}(\boldsymbol{\omega},\cdot)$ is piecewise constant in time and randomly switches at each time step. Moreover, the normalization by $\pi_m$ ensures that $\mathbb{E}[R_{rb}(\boldsymbol{\omega},t)] = R_h$ for every $t \in (0,T)$, so that $R_{rb}$ provides an unbiased approximation of the stiffness matrix $R_h$.}

    \begin{remark}[On the choice of decomposition and sampling sets]\hspace{0.1cm}
    
    Although one could use all local finite-element stiffness contributions separately, this choice may be computationally impractical. Instead, we collect these local contributions into a smaller number of groups and define each $R_{h,m}$ as the sum of the contributions belonging to the corresponding group. In this way, \eqref{eq:decom_A} holds, every $R_{h,m}$ remains symmetric and positive semidefinite, and one may choose $M\ll N$ in practice.
    
    Similarly, the random variable $\omega_k$ need not take values in the full set $\{1,\dots,2^M\}$. Instead, we may assign positive probability only to a subset of interest—usually much smaller than $2^M$—by setting $p_i = 0$ for all other values.
    
    As shown in Section~\ref{sec:numerical}, when $R_h$ arises from the discretization of a PDE on a graph, the block partition has a natural geometric interpretation: each edge contributes one block. In our numerical experiments, we cluster the edges into three groups, thereby defining three large blocks and, consequently, three submatrices $R_{h,1}$, $R_{h,2}$, and $R_{h,3}$ (i.e., $M = 3$). We then assign positive probability only to the singleton subsets
    $S_1 = \{1\}$, $S_2 = \{2\}$, and $S_3 = \{3\}$ (see Figure~\ref{fig:graph_and_dec}).
    
    This minimalist strategy ensures that $\pi_m > 0$ for all $m \in \{1,\dots,M\}$. Therefore, over each time interval $[(k-1)\delta,k\delta)$, the matrix $R_{rb}(\boldsymbol{\omega},t)$ randomly selects one of the scaled matrices $R_{h,1}/\pi_1$, $R_{h,2}/\pi_2$, or $R_{h,3}/\pi_3$.
    \end{remark}
    
    This random operator is then used in place of $R_h$ in \eqref{eq:fem_system}, leading to the random dynamical system
    \begin{align}\label{eq:random_system}
        \begin{cases}
            E_h\partial_t\y_R(\boldsymbol{\omega},t)+R_{rb}(\boldsymbol{\omega},t)\y_R(\boldsymbol{\omega},t)=\mathbf F_h, & t \in (0,T), \\
            E_h \y_R(0) =  \y^0_h,
        \end{cases}
    \end{align}
    where $\y_h^0$ and $\mathbf F_h$ are the same initial load vector and source load vector as in \eqref{eq:fem_system}, respectively.
    
    Following \cite{MR4433122}, we can prove that $\y_R(\boldsymbol{\omega}, t)$ approximates the solution $\y_h$ of equation \eqref{eq:fem_system} in expectation as $\delta \to 0$. Specifically, setting
    \begin{align}\label{eq:def_var_A}
    \text{Var}[ R_{rb}] := \sum_{i=1}^{2^M} \left( \left\|R_h - \sum_{m \in S_i} \frac{R_{h,m}}{\pi_m} \right\|^2 p_i \right),
    \end{align}
   where $\|\cdot\|$ denotes the matrix norm induced by the Euclidean norm on $\mathbb R^N$, the following theorem holds:
    
    \begin{theorem}\label{Th:RBM_1}
    Let $\y_h$ and $\y_R$ denote the solutions of \eqref{eq:fem_system} and \eqref{eq:random_system}, respectively. Then, for every $t \in (0,T)$, we have  
    \begin{align}\label{eq:estimation_2}
        &\nonumber\mathbb{E}\left[\|\y_h(t) - \y_R(t)\|^2\right]\\
        &\hspace{0.5cm}\leq C\left(\|E_h^{-1}R_h\|T^2+2T\right)\left(\|E_h^{-1}\y_h^0\|+\|E_h^{-1}\mathbf F_h\|_{L^1(0,T;\mathbb R^N)}\right)^2\frac{\operatorname{Var}[R_{rb}]}{\lambda_{\min}^2(E_h)}\,\delta,
    \end{align}
    where $\lambda_{min}(E_h)$ denotes the minimum eigenvalue of $E_h$.
    \end{theorem}
    For completeness of the exposition, we briefly outline the main steps in the proof of \Cref{Th:RBM_1}.
    \begin{proof}[Sketch of the proof of \Cref{Th:RBM_1}]
Set $\mathbf z_h=E_h^{1/2}\y_h$ and $\mathbf z_R=E_h^{1/2}\y_R$. The corresponding deterministic and randomized matrices are
\[
\widehat A_h=-E_h^{-1/2}R_hE_h^{-1/2},\qquad \widehat A_{rb}=-E_h^{-1/2}R_{rb}E_h^{-1/2}.
\]
Since $\mathbf v^\top R_{h,m}\mathbf v\geq0$ for every $\mathbf v\in\mathbb R^N$, each matrix $-E_h^{-1/2}R_{h,m}E_h^{-1/2}$ is dissipative in the Euclidean inner product. Hence the normalized systems satisfy the dissipativity requirement of \cite[Theorem 1]{MR4433122}. Moreover,
\[
\mathbb E[\widehat A_{rb}]=\widehat A_h,\qquad \mathbb E\!\left[\|\widehat A_{rb}-\widehat A_h\|^2\right]\leq\frac{\operatorname{Var}[R_{rb}]}{\lambda_{\min}^2(E_h)}.
\]
Applying \cite[Theorem 1]{MR4433122} to the normalized systems and returning to the original variables gives \eqref{eq:estimation_2}. The constant introduced when changing between the Euclidean and $E_h$-weighted norms is uniform because the condition number of the one-dimensional mass matrix $E_h$ is bounded independently of $h$.
\end{proof}

    \bigbreak
    \begin{remark}\label{remark:mass_matrix} 
    Regarding the previous theorem, the following comments are in order:
    \begin{enumerate}
         \item[1.] In \cite[Theorem 1]{MR4433122}, the state space is equipped with the Euclidean inner product. In the present FEM system, the natural inner product is the one induced by $E_h$. The change of variables $\mathbf z=E_h^{1/2}\mathbf y$ transforms this weighted inner product into the Euclidean one and transforms each matrix $-E_h^{-1}R_{h,m}$ into the dissipative matrix $-E_h^{-1/2}R_{h,m}E_h^{-1/2}$. This justifies the application of \cite[Theorem 1]{MR4433122}.
    
    \item[2.] No additional shift or coercivity assumption is required here. Indeed, the condition $\mathbf v^\top R_{h,m}\mathbf v\geq0$ implies directly that the normalized matrices $-E_h^{-1/2}R_{h,m}E_h^{-1/2}$ are dissipative.

        \item[3.] In our formulation we keep the matrices $E_h$ and $R_h$ separate, rather than rewriting \eqref{eq:fem_system} in terms of the single matrix $A_h = E_h^{-1}R_h$. In principle, \Cref{Th:RBM_1} could also be applied in this alternative setting, with the randomization acting on $A_h$ instead of $R_h$. However, in the first case the matrix being decomposed is $R_h$, while in the second it would be $E_h^{-1}R_h$. Our main motivation for decomposing $R_h$ is to preserve the analogy with the RBM for PDEs on graphs. In the graph discretization, both $\boldsymbol R_h$ and $\boldsymbol E_h$ have sparse block structures induced by the subgraphs, with off-diagonal couplings associated with the common vertex degrees of freedom. This block structure yields a natural domain decomposition of $R_h$ into submatrices associated with each subgraph. By contrast, the product $E_h^{-1}R_h$ is typically dense, destroying the block-diagonal structure and any clear domain-based interpretation. Moreover, as discussed in \Cref{remark:computational_efficiency}, working with a decomposition of $R_h$ leads to an algorithm in which, at each RBM step, it is sufficient to solve a PDE only on a single subgraph.

        \item[4.] According to \cite[Theorem 2]{MR4433122}, an analogous estimate holds for a realization-dependent input when the same input is applied to the deterministic and randomized systems and its $L^2(0,T;\mathbb R^N)$ norm is bounded uniformly with respect to the realization. Equality of the expected inputs alone is not sufficient.
    
    \end{enumerate}
    \end{remark}
    In \cite{MR4433122}, the Random Batch Method (RBM) was also adapted to address optimal control problems.
    
    Specifically, let $Q,\,D \in \R^{N\times N}$ be two symmetric and positive definite matrices, and let $\y_d\in \R^N$ be a given target. Denote by $\| \boldsymbol{v} \|_{Q}^2 := \langle \boldsymbol{v}, Q \boldsymbol{v} \rangle$ and $\|\boldsymbol{v}\|_D^2 := \langle \boldsymbol{v}, D \boldsymbol{v} \rangle$ for $\boldsymbol{v} \in \R^N$, and consider the optimal control problem
    \begin{align}\label{eq:dis_optimal_control}
        \min_{{\f_h} \in L^2(0,T;\mathbb{R}^N)} \left\{ J_{D,Q}({\f_h})= \frac{1}{2} \int_0^T \left( \| {\f_h}(t)\|^2_{D} + \|\y_h(t) - \y_d\|^2_{Q} \right) dt\right\},
    \end{align}
    where $\y_h$ satisfies $E_h\partial_t\y_h+R_h\y_h=E_h\f_h$ and $E_h\y_h(0)=\y_h^0$, so that $\f_h$ is the coefficient vector of the distributed control. A unique optimal pair $(\y_h,{\f_h})$ solving \eqref{eq:dis_optimal_control} exists (\cite[Section 1.4]{MR2583281}). 
    
    Additionally, consider the $\bomega$-dependent optimal control problem
    \begin{align}\label{eq:random_optimal_control}
        \min_{{\f_R}(\bomega) \in L^2(0,T;\mathbb{R}^N)} \left\{ J^R_{D,Q}({\f_R}(\bomega))= \frac{1}{2} \int_0^T \left( \| {\f_R}(t)\|^2_{D} + \|\y_R(t) - \y_d\|^2_{Q} \right) dt\right\},
    \end{align}
    where $\y_R$ satisfies $E_h\partial_t\y_R+R_{rb}\y_R=E_h\f_R$ and $E_h\y_R(0)=\y_h^0$, so that $\f_R$ is the coefficient vector of the randomized optimal control. For each realization of $\bomega$, a unique optimal pair $(\y_R(\bomega),\f_R(\bomega))$ exists.
    
     Combining \cite[Theorems 2 and 4]{MR4433122}, the following holds:
    \begin{theorem}\label{th:rbm_optimal_control}
Let $(\y_h,\f_h)$ and $(\y_R,\f_R)$ denote the optimal pairs of \eqref{eq:dis_optimal_control} and \eqref{eq:random_optimal_control}, respectively. Then there exists a constant $C_{oc,h}>0$, independent of $\delta$, such that
\begin{align}\label{eq:estimation_oc_fin}
\mathbb E\!\left[\|\f_h-\f_R\|_{L^2(0,T;\mathbb R^N)}^2+\|\y_h-\y_R\|_{L^2(0,T;\mathbb R^N)}^2\right]\leq C_{oc,h}\frac{\operatorname{Var}[R_{rb}]}{\lambda_{\min}^2(E_h)}\,\delta.
\end{align}
\end{theorem}
    \begin{remark}[On the tracking of the constant $C_{oc,h}$]
In the case $D=Q=E_h$ considered below, set
\[
\mathbf z=E_h^{1/2}\mathbf y,\quad \mathbf v=E_h^{1/2}\mathbf f,\quad \widehat A_h=-E_h^{-1/2}R_hE_h^{-1/2},\quad \widehat A_{rb}=-E_h^{-1/2}R_{rb}E_h^{-1/2}.
\]
The dynamics and cost functional are then transformed into
\[
\partial_t\mathbf z=\widehat A_h\mathbf z+\mathbf v,\qquad \frac12\int_0^T\left(\|\mathbf v(t)\|^2+\|\mathbf z(t)-\mathbf z_d\|^2\right)\,dt,
\]
and analogously for the randomized problem. Hence the strong-convexity constant in \cite[Theorem 4]{MR4433122} is equal to one and is independent of $h$. Tracing the proof of \cite[Lemma 2 and Theorem 4]{MR4433122} gives
\[
\mathbb E\!\left[\|\mathbf v_R-\mathbf v_h\|_{L^2(0,T;\mathbb R^N)}^2\right]\leq C\delta\operatorname{Var}[\widehat A_{rb}]\left(1+\|\widehat A_h\|T^2+2T\right),
\]
where $C$ depends on the problem data and $T$, but is independent of $h$ and $\delta$. The corresponding state estimate follows by combining \cite[Theorem 2]{MR4433122} with continuous dependence of the deterministic dynamics on the control. Since the normalized Euclidean norms coincide with the physical $L^2(0,L)$ norms, \eqref{eq:estimation_var_a} and \eqref{eq:min_eigen_E} imply
\[
\mathbb E\!\left[\|f_R-f_h\|_{L^2(0,T;L^2(0,L))}^2+\|y_R-y_h\|_{L^2(0,T;L^2(0,L))}^2\right]\leq C\frac{\delta}{h^6}C(M).
\]
\end{remark}

    \section{Approximation by the Random Batch Method}\label{sec:main_results}
    {In this section we combine the FEM semi-discretization from Section~\ref{section_FEM} with the RBM framework of Section~\ref{sec:rbm} for the heat equation on an interval.}
    \subsection{Approximation of the Heat Equation via the RBM}
    Let us consider the basis functions $\{\phi_j\}_{j=1}^N$ of the space $V_{h}$ introduced in \Cref{section_FEM}. Let $\y_R(\bomega, t) \in \R^N$ be the solution of \eqref{eq:random_system}, and denote by
    \begin{align}\label{eq:definition_yR}
    y_R(x,t,\bomega) := \sum_{j=1}^N(\y_R)_j(\bomega,t)\phi_j(x),
    \end{align}
    where $(\y_R)_j$ denotes the $j$-th coordinate of $\y_R(\bomega, t)$ on the basis of $V_{h}$. 
    
    In the following, to simplify the presentation, we assume that $R_h=R/h$ and $R_{h,m}=R_m/h$ for $m\in\{1,\dots,M\}$, where the corresponding dimensionless matrices satisfy
\[
C(M):=\sup_h\sum_{i=1}^{2^M}p_i\left\|R-\sum_{m\in S_i}\frac{R_m}{\pi_m}\right\|^2<\infty.
\]
This is the scaling naturally inherited from the stiffness matrix $R_h$. In this setting, the following result holds:
    
    \begin{theorem}\label{th:convergence_RBM_FE}
    Let $y_R$ be as in \eqref{eq:definition_yR} and let $y$ be the solution of \eqref{eq:heat_eq} with initial condition $y^0\in H^2(0,L)\cap H_0^1(0,L)$ and source term $f\in C^{0,\theta}([0,T];L^2(0,L))$ for some $0<\theta<1$. Then, for $h < \sqrt{6T}$ given, we have
    \begin{align}\label{eq:estim_app_heat}
    \E[\|y_R(t)-y(t)\|_{L^2(0,L)}^2]\leq C\left(h^4+\frac{\delta}{h^6}C(M)\right),
    \end{align}
    for every $t \in (0,T)$, where $C$ and $C(M)$ are two positive constants independent of $h$ and $\delta$, with $C(M)$ depending on the chosen decomposition and $\delta$ being the RBM switching parameter introduced in \Cref{sec:rbm}.
    \end{theorem}
    \begin{remark}
        Assuming that $C(M)$ remains bounded, estimate \eqref{eq:estim_app_heat} proves convergence whenever $\delta=o(h^6)$. In particular, if $\delta=h^{6+\alpha}$ for some $\alpha>0$, the randomization contribution to the mean-square error is of order $h^\alpha$, and the total mean-square error is of order $h^{\min\{4,\alpha\}}$. This is illustrated in \Cref{sec:overlapping_dec}.
    \end{remark}
    
    \begin{proof}
From the scaling assumption,
\begin{align}\label{eq:estimation_var_a}
\operatorname{Var}[R_{rb}]=\sum_{i=1}^{2^M}p_i\left\|R_h-\sum_{m\in S_i}\frac{R_{h,m}}{\pi_m}\right\|^2\leq\frac{C(M)}{h^2}.
\end{align}
Since $E_h$ is a tridiagonal Toeplitz matrix,
\begin{align}\label{eq:min_eigen_E}
\lambda_{\min}(E_h)=h\left(\frac23+\frac13\cos\left(\frac{N\pi}{N+1}\right)\right)>\frac h3.
\end{align}
Set $\mathbf z_h=E_h^{1/2}\y_h$, $\mathbf z_R=E_h^{1/2}\y_R$, and 
\[ 
\widehat A_h=-E_h^{-1/2}R_hE_h^{-1/2},\qquad \widehat A_{rb}=-E_h^{-1/2}R_{rb}E_h^{-1/2}.
\]
Using \eqref{eq:estimation_var_a} and \eqref{eq:min_eigen_E}, we obtain
\[
\mathbb E\!\left[\|\widehat A_{rb}-\widehat A_h\|^2\right]\leq\|E_h^{-1}\|^2\operatorname{Var}[R_{rb}]\leq\frac{9C(M)}{h^4}.
\]
Moreover, since $\|R_h\|\leq4/h$,
\[
\|\widehat A_h\|\leq\|E_h^{-1}\|\|R_h\|\leq\frac{12}{h^2}.
\]
Let $\mathbf g_h=E_h^{-1/2}\mathbf F_h$. Since $\mathbf F_h$ is the load vector associated with $f$,
\begin{align*}
    \|\mathbf z_h(0)\|=\|P_hy^0\|_{L^2(0,L)}\leq\|y^0\|_{L^2(0,L)},\\ \|\mathbf g_h(t)\|=\|P_hf(t)\|_{L^2(0,L)}\leq\|f(t)\|_{L^2(0,L)}.
\end{align*}
Applying \cite[Theorem 1]{MR4433122} to the normalized systems gives
\[
\mathbb E\!\left[\|\mathbf z_R(t)-\mathbf z_h(t)\|^2\right]\leq C\frac{\delta}{h^6}C(M).
\]
Since
\[
\|\mathbf z_R(t)-\mathbf z_h(t)\|^2=\|y_R(t)-y_h(t)\|_{L^2(0,L)}^2,
\]
the triangle inequality and \Cref{th:fin_elem1} yield to \eqref{eq:estim_app_heat}.
\end{proof}

    \begin{remark}\label{remark:conv_ht_hx}
    Several remarks are in order:
    \begin{enumerate}
            \item[1.]{The factor $h^{-6}$ in \eqref{eq:estim_app_heat} results from the estimates $\operatorname{Var}[\widehat A_{rb}]=O(h^{-4})$ and $\|\widehat A_h\|=O(h^{-2})$ in the mass-normalized variables. This exponent is a consequence of the operator-norm estimate in \cite[Theorem 1]{MR4433122} and should not be interpreted as an intrinsic or sharp property of the RBM applied to parabolic PDEs.}

        \item[2.]{When $h\to0$, the semi-discrete FEM solution approaches the continuous PDE solution. If $\delta>0$ is kept fixed, the randomization term in \eqref{eq:estim_app_heat} may dominate. Assuming that $C(M)$ remains bounded, the estimate proves convergence whenever $\delta=o(h^6)$. If the RBM contribution is required to be no larger than the squared FEM error $O(h^4)$, it is sufficient to impose $\delta=O(h^{10})$. These are sufficient coupling conditions arising from the present a priori estimate and should not be interpreted as CFL-type stability restrictions.}
    
        \item[3.] The multiplicative constant $C(M)$ in \eqref{eq:estimation_var_a} plays a relevant role in the error estimate  \eqref{eq:estim_app_heat}, since it amplifies the error term   $\delta/h^6$. Consequently, minimizing $C(M)$ is a relevant objective in our setting, as it directly impacts the overall efficiency of the method. Moreover, the value of $C(M)$ is closely linked to the way the matrix $R_h$ is decomposed -- an aspect we further elaborate on in \Cref{sec:c_and_o}. In particular, when $M=1$ and the unique matrix is selected with probability one, the RBM reduces to the standard FEM approximation and $C(M)=0$.
    
        \item[4.] By Theorem \ref{th:convergence_RBM_FE} we have that the variance 
        \begin{align*}
            \mathbb{V}[\|y_R(t)-y(t)\|]&:=\mathbb{E}[\|y_R(t)-y(t)\|^2] -\left(\mathbb{E}[\|y_R(t)-y(t)\|]\right)^2\\
            &\leq C\left(h^4 + \frac{\delta}{h^6}C(M)\right).
        \end{align*}
        Moreover, Markov’s inequality yields, for any $\varepsilon>0$,
        \begin{align*}
            \mathbb{P}(\|y_R(t)-y(t)\|^2>\varepsilon)\leq \frac{C}{\varepsilon}\left(h^4 + \frac{\delta}{h^6}C(M)\right).
        \end{align*}
        Therefore, whenever $C\left(h^4 + \delta/h^6C(M)\right)$ is small, not only the mean‐square error but also the variance and the probability of large deviations are uniformly small. In this regime, a single realization is close to the exact solution with high probability.
        
        \item[5.] The analysis in Theorems \ref{th:convergence_RBM_FE} and \ref{th:conver_optimal_control} (below) can be adapted and applied to the finite difference method (FDM). However, obtaining estimates similar to \eqref{eq:estimation_finite_elements} with the FDM, requires smoother data to ensure that $y \in L^2(0,T;C^4(0,L))$ (see, for instance,  \cite[Theorem 16.5]{MR2478556}, where the convergence error for the FDM depends on $\|\partial_x^4 y\|_{L^2(0,T;C(0,L))}$).
    \end{enumerate}
    \end{remark}

    \subsection{Application to Optimal Control}\label{eq:discrete_ocp}
    In this section, we prove the convergence of the RBM in the context of optimal control for the heat equation; we employ a discrete approach. Namely, we first discretize the PDE to later apply the RBM. Let $y_d\in L^2(0,L)$ be a given target, and consider the optimal control problem 
    \begin{align}\label{eq:optimal_control}
          \min_{f \in L^2(0,T;L^2(0,L))} \left\{ J(f)= \frac{1}{2} \int_0^T \left( \|f(t)\|^2_{L^2(0,L)} + \|y(t) - y_d\|_{L^2(0,L)}^2 \right) dt\right\},
    \end{align}
    where  $y$ is the solution of \eqref{eq:heat_eq} with right-hand side $f$. According to \cite[Chapter 3]{MR2583281}, there exists a unique optimal pair $(y, f)$ for \eqref{eq:optimal_control}. To approximate the optimal control problem \eqref{eq:optimal_control}, we can consider the finite-dimensional problem 
    \begin{align}\label{eq:finite_dim_ffunctional}
      \min_{f_h\in L^2(0,T;L^2(0,L))} \left\{ J_h(f_h)= \frac12\int_0^T \left(\|f_h(t)\|^2_{L^2(0,L)}+\|y_h(t)-y_d \|^2_{L^2(0,L)} \right) dt \right\},
    \end{align}
    where $y_h$ is the solution of \eqref{eq:variational_fem} with right-hand side $f_h$. The convergence of the FEM in this context, is assured by the following result in \cite{MR1119274}.
    
    \begin{theorem}\label{th:conver_optimal_control}
Let $(y,f)$ denote the optimal pair of \eqref{eq:optimal_control} with $y^0,y_d\in L^2(0,L)$, and let $(y_h,f_h)$ denote the optimal pair of \eqref{eq:finite_dim_ffunctional}. Then there exists a constant $C>0$, independent of $h$, such that
\begin{align*}
\|y_h(t)-y(t)\|_{L^2(0,L)}^2+\|f_h(t)-f(t)\|_{L^2(0,L)}^2\leq Ch^4\left(\frac{T}{t(T-t)}+\ln\left(\frac1h\right)\right)^2
\end{align*}
for every $t\in(0,T)$.
\end{theorem}
\begin{proof}
Let $H=L^2(0,L)$, let $A=-\partial_{xx}$ with $D(A)=H^2(0,L)\cap H_0^1(0,L)$, and let $P_h:H\to V_h$ be the $L^2$-orthogonal projection. By \cite[Theorem 4.1 and Remark 4.2]{FujitaMizutani1976}, the continuous and discrete heat semigroups satisfy
\[
\|T(t-s)-T_h(t-s)P_h\|_{\mathcal L(H)}\leq Ch^2(t-s)^{-1},\qquad 0\leq s<t\leq T.
\]
Moreover, since $A$ and its Galerkin approximation are self-adjoint and nonnegative, $\|T(t-s)\|_{\mathcal L(H)}\leq1$ and $\|T_h(t-s)\|_{\mathcal L(V_h)}\leq1.$ Thus the present setting satisfies conditions (A1) and (A2) of \cite{MR1119274}. To include the tracking term, consider the augmented space $\widetilde H=H\times\mathbb R$, the state $(y,1)$, and the operator $\widetilde D(z,\alpha)=\left(z-\alpha y_d,-(y_d,z)_H+\alpha\|y_d\|_H^2\right).$
This operator is bounded, self-adjoint and nonnegative for every $y_d\in L^2(0,L)$, since $(\widetilde D(z,\alpha),(z,\alpha))_{\widetilde H}=\|z-\alpha y_d\|_H^2.$ The augmented continuous and discrete semigroups inherit (A1) and (A2), because the additional scalar component is represented exactly. Therefore \cite[Theorem 3.1]{MR1119274} applies to the Riccati operators and \cite[Theorem 3.2]{MR1119274} applies to the optimal trajectories. The feedback representation of the optimal controls then gives the stated estimate.
\end{proof}
\begin{corollary}\label{cor:optimal_fem_l2}
Under the assumptions of \Cref{th:conver_optimal_control},
\[
\|y_h-y\|_{L^2(0,T;L^2(0,L))}^2+\|f_h-f\|_{L^2(0,T;L^2(0,L))}^2\leq Ch^2.
\]
\end{corollary}

\begin{proof}
The feedback representations and the uniform estimates in \cite[Lemma 3.4]{MR1119274} imply that the continuous and discrete optimal states and controls are uniformly bounded with respect to $h$. We use this uniform bound on $(0,h^2)\cup(T-h^2,T)$ and integrate the estimate of \Cref{th:conver_optimal_control} on $(h^2,T-h^2)$. Since
\[
h^4\int_{h^2}^{T-h^2}\left(t^{-2}+(T-t)^{-2}+\ln^2(1/h)+1\right)\,dt\leq Ch^2,
\]
the result follows.
\end{proof}

    Denote by $(\y_h, \f_h)$ the vectors of coefficients of $(y_h, f_h)$. Observe that 
    \begin{align}\label{eq:characterization_norm_l2}
        \|y_h(t)\|^2_{L^2(0,L)} = \left\|\sum_{j=1}^N y_j(t) \phi_j\right\|^2_{L^2(0,L)} = \|\y_h\|^2_{E_h} = \langle E_h \y_h, \y_h \rangle,
    \end{align}
    where $E_h$ corresponds to the mass matrix introduced in \eqref{eq:stiffness_mass_matrices}. Therefore, when $D=Q=E_h$, the coefficient vectors $(\y_h,\f_h)$ solve the finite-dimensional problem \eqref{eq:dis_optimal_control}, where the controlled dynamics is $E_h\partial_t\y_h+R_h\y_h=E_h\f_h$ and $\y_d$ is the coefficient vector of $P_hy_d$.
    
    Let $(\y_R,\f_R)$ denote the optimal pair of \eqref{eq:random_optimal_control}, with controlled dynamics $E_h\partial_t\y_R+R_{rb}\y_R=E_h\f_R$, and let $(y_R,f_R)$ be the corresponding functions in $V_h$. Then, we have the following theorem.
    \begin{theorem}\label{th:op_random_approximation}
Let $(y,f)$ denote the optimal pair of \eqref{eq:optimal_control} with $y^0,y_d\in L^2(0,L)$, and let $(y_R,f_R)$ be as given above. Then there exists a constant $C>0$, independent of $h$ and $\delta$, such that
\begin{align}\label{eq:estim_opt_control_heat_rand}
\mathbb E\!\left[\|y_R-y\|_{L^2(0,T;L^2(0,L))}^2+\|f_R-f\|_{L^2(0,T;L^2(0,L))}^2\right]\leq C\left(\frac{\delta}{h^6}C(M)+h^2\right).
\end{align}
\end{theorem}
    
    \begin{proof}
By the triangle inequality,
\begin{align*}
    \|y_R-y&\|_{L^2(0,T;L^2)}^2+\|f_R-f\|_{L^2(0,T;L^2)}^2\leq2(\|y_R-y_h\|_{L^2(0,T;L^2)}^2\\&+\|f_R-f_h\|_{L^2(0,T;L^2)}^2)
+2(\|y_h-y\|_{L^2(0,T;L^2)}^2+\|f_h-f\|_{L^2(0,T;L^2)}^2).
\end{align*}
The second pair of terms is bounded by $Ch^2$ according to \Cref{cor:optimal_fem_l2}. The estimate derived in the remark on $C_{oc,h}$ gives
\[
\mathbb E\!\left[\|y_R-y_h\|_{L^2(0,T;L^2)}^2+\|f_R-f_h\|_{L^2(0,T;L^2)}^2\right]\leq C\frac{\delta}{h^6}C(M).
\]
Combining both estimates proves \eqref{eq:estim_opt_control_heat_rand}.
\end{proof}
    
    \begin{remark} 
        Note the following:
        \begin{enumerate}
           \item[1.]{The power $h^{-6}$ in \eqref{eq:estim_opt_control_heat_rand} results from the estimates $\operatorname{Var}[\widehat A_{rb}]=O(h^{-4})$ and $\|\widehat A_h\|=O(h^{-2})$ in mass-normalized variables. It is a consequence of the operator-norm estimates in \cite{MR4433122} and should not be regarded as an intrinsic or necessarily sharp dependence of the RBM error on the spatial mesh size.}

\item[2.]{Assuming that $C(M)$ remains bounded, estimate \eqref{eq:estim_opt_control_heat_rand} proves convergence whenever $\delta=o(h^6)$. If the RBM contribution is required to be no larger than the global squared FEM error $O(h^2)$, it is sufficient to impose $\delta=O(h^8)$.}

\item[3.]{The logarithmic factor in \Cref{th:conver_optimal_control} originates from \cite[Theorems 3.1 and 3.2]{MR1119274}. On every fixed interior interval $[\varepsilon,T-\varepsilon]$, it leads to the squared error $O(h^4(1+\ln(1/h))^2)$. On the full interval $(0,T)$, the singular endpoint factors are handled using uniform stability, which yields the global $O(h^2)$ squared estimate of \Cref{cor:optimal_fem_l2}.}
        \end{enumerate}
    \end{remark}

    \section{The RBM for the Heat Equation on Networks}\label{sec:numerical}
    {In this section, we present the heat equation on graphs and discuss its discretization and the application of the RBM. After FEM discretization on each edge and incorporation of the coupling conditions at the vertices, the network problem again reduces to a finite-dimensional system of the form
    \[
    \boldsymbol{E}_h \partial_t \Y_h(t) + \boldsymbol{R}_h \Y_h(t) = \F_h(t),
    \]
    with symmetric positive definite mass and stiffness matrices. Thus, at the matrix level, the situation is completely analogous to the interval case analysed in Section~\ref{sec:main_results}; the graph topology is reflected in the block structure of $\boldsymbol{R}_h$, but the RBM construction and the dependence of the error bounds on $h$ and $\delta$ remain of the same type.}
    
    \subsection{The Heat Equation on Graphs}\label{subsec:numerics_equation}
    Let $\mathcal{G} := (\mathcal{E}, \mathcal{V})$ be an oriented and connected metric graph with vertices $v \in \mathcal{V}$ and edges $e \in \mathcal{E}$. In the following, to simplify the presentation, we will focus on the graph $\mathcal{G}$ illustrated in \Cref{fig:graph_and_dec}. Nonetheless, the methodology developed herein is applicable to any connected metric graph. The graph $\mathcal{G}$ is composed of the edges and vertices:
    \begin{align*}
        \mathcal{E} = \{e_i\}_{i=1}^{10},\quad \text{and} \quad \mathcal{V} = \mathcal{V}_0 \cup \mathcal{V}_b=\{v_i\}_{i=1}^3\cup  \{v_i\}_{i=4}^{11},
    \end{align*}
    where $\mathcal{V}_0$ and $\mathcal{V}_b$ denote the interior and boundary vertices, respectively. Here, the orientation of the graph is used merely to identify each edge $e \in \mathcal{E}$ with an interval $(0, l_e)$ for $l_e > 0$, without ambiguity, where $0$ and $l_e$ represent the incoming and outgoing vertices of $e$, respectively. For simplicity, we also assume that $l_e = L > 0$ for all $e \in \mathcal{E}$.

    \begin{figure}[H]
        \centering
    \includegraphics[width=0.9\linewidth]{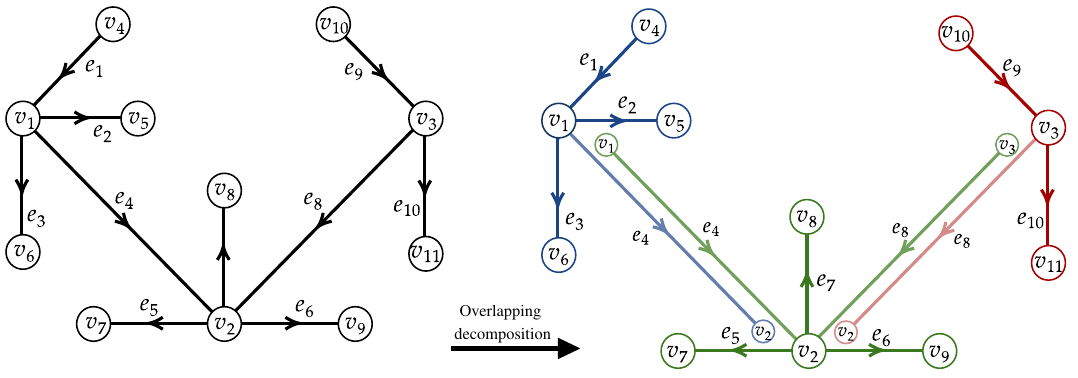}
        \caption{Illustration of a graph decomposition into three edge groups (shown in different colors).  
    Each group is associated with one block of the stiffness matrix. Since edges $e_4$ and $e_8$ belong to two groups, the chosen block overlaps.}
        \label{fig:graph_and_dec}
    \end{figure}
    
    On each edge $e_i$ with $i \in \{1, \dots, 10\}$, we consider the heat equation
    \begin{align}\label{eq:heat_net}
    \begin{cases}
    \pt y^{e_i} - \pxx y^{e_i} = f^{e_i},\quad &(x,t) \in (0,L) \times (0,T),\\
    y^{e_i}(x,0) = y_{e_i}^0(x),&x \in (0,L),
    \end{cases}
    \end{align}
    where $y_{e_i}^0$ is the initial condition and $f^{e_i}$ is the source term, both defined for every $i \in \{1, \dots, 10\}$. System \eqref{eq:heat_net} is complemented with boundary and coupling conditions:
    \begin{align}\label{eq:heat_bc}
    \begin{cases}
    y^e(v,t) = 0, \quad & v \in \mathcal{V}_b,\, e \in \mathcal{E}(v),\\
    y^{e'}(v,t) = y^{e''}(v,t),& v \in \mathcal{V}_0,\, e', e'' \in \mathcal{E}(v),\\
    \displaystyle\sum_{e \in \mathcal{E}(v)} \px y^e(v,t) n_e(v) = 0, & v \in \mathcal{V}_0,
    \end{cases}
    \end{align}
    over the time interval $(0,T)$. Here, $\mathcal{E}(v)$ denotes the set of edges incident to $v$, and $n_e(v)$ is the corresponding incidence sign: $n_e(v)=-1$ when $v$ is the initial endpoint of $e$ and $n_e(v)=1$ when it is the terminal endpoint. The first equation in \eqref{eq:heat_bc} corresponds to a natural extension of the null Dirichlet boundary condition. The second and third conditions of \eqref{eq:heat_bc} are coupling conditions for the interior vertices of the graph. In the following, we denote by $f$ and $y^0$ the families $\{f^{e_i}\}_{i=1}^{10}$ and $\{y_{e_i}^0\}_{i=1}^{10}$, respectively. We use $H_0^1(\mathcal E)$ to denote the space of edgewise $H^1$ functions that are continuous at the interior vertices and vanish at the boundary vertices. We assume that $f\in L^2(0,T;L^2(\mathcal E))$ and $y^0\in H_0^1(\mathcal E)$; see \cite[Chapter 3]{MR3243602}. The well-posedness results in \cite{Egger} ensure the existence of a unique solution $y\in C([0,T];L^2(\mathcal E))\cap L^2(0,T;H_0^1(\mathcal E))$.
    \smallbreak
    
     For the numerical example considered below, we set $L=1$ and take the initial condition $y^0$ and the source term $f$ as specified in \Cref{tab:expressions_ei}. The sequence of coefficients \( \{p_{e_i}\}_{i=1}^{10} \), defined in \Cref{tab:coefficients_pe}, is carefully chosen so that $y^e$ satisfies \eqref{eq:heat_net}–\eqref{eq:heat_bc} with the corresponding data. This construction follows the method of manufactured solutions, as described in \cite[Chapter 12]{MR3931345}. That is, we begin by prescribing a smooth candidate function $y^e$ as the exact solution and then define $f$ and $y^0$ accordingly, ensuring that the equation and boundary conditions are satisfied. This approach allows for a controlled setting in which the numerical error can be precisely assessed.

    \begin{table}[H]
      \centering
      \caption{Expressions for \( y_{e_i}^0(x) \), \( f^{e_i}(x,t) \) and  \( y^{e_i}(x,t) \) for \( i \in \{1, \dots, 10\} \).}
      \label{tab:expressions_ei}
      \begin{tabular}{ccccc}
        \toprule
          \( y_{e_i}^0(x) \)& & \( f^{e_i}(x,t) \) && \( y^{e_i}(x,t) \) \\
        \midrule
         \( p_{e_i} \cdot x \cdot (1 - x) \) & &\( p_{e_i} \cdot (2 - x + x^2) \cdot e^{-t} \) & & \( p_{e_i} \cdot x \cdot (1 - x) \cdot e^{-t} \) \\
        \bottomrule
      \end{tabular}
    \end{table}

    \begin{table}[H]
      \centering
      \caption{Coefficients \( p_{e_i} \) for \( i \in \{1, \dots, 10\} \) used to introduce the problem data.}
      \label{tab:coefficients_pe}
      \begin{tabular}{cccccccccc}
        \toprule
        \( p_{e_1} \) & \( p_{e_2} \) & \( p_{e_3} \) & \( p_{e_4} \) & \( p_{e_5} \) & \( p_{e_6} \) & \( p_{e_7} \) & \( p_{e_8} \) & \( p_{e_9} \) & \( p_{e_{10}} \) \\
        \midrule
        1 & -1 & -1 & 1 & -1 & -1 & -1 & 2 & -1 & -1 \\
        \bottomrule
      \end{tabular}
    \end{table}
    
    We will use finite elements to introduce a semi-discrete version of the system \eqref{eq:heat_net}-\eqref{eq:heat_bc}. For this purpose, it is necessary to introduce a finite-dimensional space $V_h^{\mathcal{E}}$ to approximate  $H_0^1(\mathcal{E})$. On each edge $e_i$ for $i \in \{1, \dots, 10\}$, we introduce the spatial mesh $\{x_j^i\}_{j=0}^{N+1}$ with step size $h$. Due to the structure of the graph, we identify the vertices as follows: 
    \begin{align*}
        x_{N+1}^1 = x_0^2 = x_0^3 = x_0^4,\quad x_{N+1}^4 = x_0^5 = x_0^6 = x_0^7 = x_{N+1}^8,\quad x_{N+1}^9 = x_0^8 = x_0^{10}.
    \end{align*}
    On each edge $e_i$ we can introduce the basis functions $\{\phi_j^i\}_{j=1}^N$ as in \Cref{section_FEM}. These functions span a finite-dimensional space that can be used to approximate the solution in the interior of each edge. However, we also need to introduce basis functions to approximate the solution on the graph's interior vertices, or junctions. For this purpose, writing $\left.\phi\right|_{e_i}$ for the restriction of a function $\phi$ to the edge $e_i$, we define
\begin{align*}
\left.\phi_0^4\right|_{e_i}(x)&=
\begin{cases}
(x-x_N^1)/h, & i=1,\quad x\in[x_N^1,x_{N+1}^1],\\
(x_1^i-x)/h, & i\in\{2,3,4\},\quad x\in[x_0^i,x_1^i],\\
0, & \text{otherwise},
\end{cases}\\
\left.\phi_0^6\right|_{e_i}(x)&=
\begin{cases}
(x-x_N^i)/h, & i\in\{4,8\},\quad x\in[x_N^i,x_{N+1}^i],\\
(x_1^i-x)/h, & i\in\{5,6,7\},\quad x\in[x_0^i,x_1^i],\\
0, & \text{otherwise},
\end{cases}\\
\left.\phi_0^{10}\right|_{e_i}(x)&=
\begin{cases}
(x-x_N^9)/h, & i=9,\quad x\in[x_N^9,x_{N+1}^9],\\
(x_1^i-x)/h, & i\in\{8,10\},\quad x\in[x_0^i,x_1^i],\\
0, & \text{otherwise}.
\end{cases}
\end{align*}
The basis functions $\phi_0^4$ and $\phi_0^6$ are illustrated in \Cref{fig:basis_illustration}.
    \begin{figure}[H]
        \centering
        \includegraphics[scale=0.6]{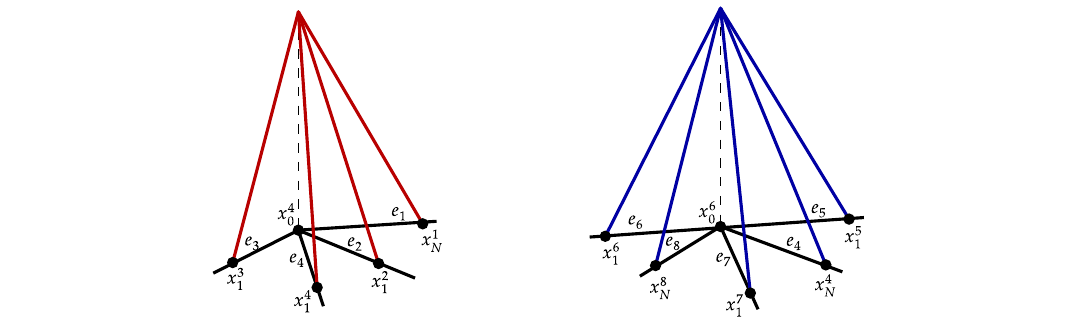}
        \caption{Illustration of the basis functions defined at the interior vertices $x^4_0$ and $x^6_0$ of the graph. Both functions attain a maximum value of 1.}
        \label{fig:basis_illustration}
    \end{figure}
    Let us consider the space $V_{h}^{\mathcal{E}}$ as the span of the functions $\{\phi_0^4, \phi_0^6, \phi_0^{10}\} \cup \bigcup_{i=1}^{10} \bigcup_{j=1}^N \{\phi_j^i\}$. Observe that the finite-dimensional space $V_{h}^{\mathcal{E}}$ is a natural finite element approximation of the space $H_0^1(\mathcal{E})$. Let $y_h^0=P_hy^0$ be the $L^2(\mathcal E)$-orthogonal projection of $y^0$ onto $V_h^{\mathcal E}$. The semidiscrete weak formulation of \eqref{eq:heat_net}--\eqref{eq:heat_bc} reads as follows: Find $y_h\in H^1(0,T;V_h^{\mathcal E})$ such that, for every $\varphi_h\in V_h^{\mathcal E}$ and almost every $t\in(0,T)$,
\begin{align}\label{eq:variational_relation_y_graph}
\sum_{i=1}^{10}\int_0^L\left(\partial_t y_h^{e_i}(x,t)\varphi_h^{e_i}(x)+\partial_x y_h^{e_i}(x,t)\partial_x\varphi_h^{e_i}(x)\right)\,dx=\sum_{i=1}^{10}\int_0^L f^{e_i}(x,t)\varphi_h^{e_i}(x)\,dx,
\end{align}
with $y_h(0)=y_h^0$. The global semidiscrete solution can be written as
\begin{align}\label{eq:discretization_y_net}
y_h(x,t)=y_0^{e_4}(t)\phi_0^4(x)+y_0^{e_6}(t)\phi_0^6(x)+y_0^{e_{10}}(t)\phi_0^{10}(x)+\sum_{i=1}^{10}\sum_{j=1}^N y_j^{e_i}(t)\phi_j^i(x).
\end{align}
    We use the convention of storing the coefficients associated with the interior vertices $v_1$, $v_2$, and $v_3$ in the blocks corresponding to $e_4$, $e_6$, and $e_{10}$, respectively. This is only an ordering convention: each vertex basis function contributes on every edge incident to that vertex.
    
    Then, let us consider the vector of coefficients $\Y_h(t) = \left(\y^{e_1}_h(t), \dots, \y^{e_{10}}_h(t)\right)^\top$, where
    \begin{align}\label{eq:notation_vector_sol}
    \begin{cases}
        \y^{e_i}_h = (y^{e_i}_0(t), \dots, y^{e_i}_N(t))^\top, & \text{if } i \in \{4,6,10\}, \\
        \y^{e_i}_h = (y^{e_i}_1(t), \dots, y^{e_i}_N(t))^\top, & \text{otherwise}.
    \end{cases}
    \end{align}
    Here, $\Y_h(t) \in \R^{10N+3}$ (with $N$ discretization points for each of the $10$ edges and $3$ extra discretization points for the internal vertices). Let $\{\psi_\ell\}_{\ell=1}^{10N+3}$ denote the global nodal basis of $V_h^{\mathcal E}$ in the ordering induced by $\Y_h$. The assembled mass and stiffness matrices and the load vectors are defined by
\begin{align*}
(\boldsymbol E_h)_{\ell k}&=\sum_{i=1}^{10}\int_0^L\psi_k^{e_i}(x)\psi_\ell^{e_i}(x)\,dx,\qquad
(\boldsymbol R_h)_{\ell k}=\sum_{i=1}^{10}\int_0^L\partial_x\psi_k^{e_i}(x)\partial_x\psi_\ell^{e_i}(x)\,dx,\\
(\F_h(t))_\ell&=\sum_{i=1}^{10}\int_0^L f^{e_i}(x,t)\psi_\ell^{e_i}(x)\,dx,\qquad
(\boldsymbol E_h\mathbf y_h^0)_\ell=\sum_{i=1}^{10}\int_0^L y_{e_i}^0(x)\psi_\ell^{e_i}(x)\,dx.
\end{align*}
Observe that $\Y_h$ satisfies the following system:
    \begin{align}\label{eq:discrete_network_fem}\begin{cases}
            \boldsymbol{E}_h\pt \Y_h + \boldsymbol{R}_h\Y_h = \F_h, & t \in (0,T), \\
             \boldsymbol{E}_h\Y_h(0) = \Y_h^0.
        \end{cases}
    \end{align}
    The matrix $\boldsymbol{R}_h$ is given by
    \begin{align}
    &\hspace{0.2cm}\bigl[\hspace{0.4cm}
      e_1 \quad e_2 \quad e_3 \hspace{0.72cm} v_1 \hspace{0.74cm} e_4 \quad e_5 \hspace{0.7cm}v_2 \hspace{0.72cm}
      e_6 \quad e_7 \quad e_8 \quad e_9 \hspace{0.81cm} v_3\hspace{0.74cm} e_{10}\hspace{0.4cm}
    \bigr]\nonumber\\
    &
    \begin{pmatrix}
      \begin{array}{c|c|c|c|c|c|c}
        %
        \begin{array}{c}
    \boldsymbol{R}_h^1
        \end{array}
        &
        \begin{array}{c}
          c_1^4
        \end{array}
        &
        \begin{array}{c}
          0
        \end{array}
        &
        \begin{array}{c}
          0
        \end{array}
        &
        \begin{array}{c}
          0
        \end{array}
        &
        \begin{array}{c}
          0
        \end{array}
        &
        \begin{array}{c}
          0
        \end{array}
        \\[1.5mm]
        \hline
        %
        \begin{array}{c}
          a_1^4 \quad a_2^4 \quad a_3^4
        \end{array}
        &
        \begin{array}{c}
           \frac{4}{h}
        \end{array}
        &
        \begin{array}{c}
          a_4^4 \quad 0
        \end{array}
        &
        \begin{array}{c}
          0
        \end{array}
        &
        \begin{array}{c}
          0
        \end{array}
        &
        \begin{array}{c}
          0
        \end{array}
        &
        \begin{array}{c}
          0
        \end{array}
        \\[1.5mm]
        \hline
        %
        \begin{array}{c}
          0
        \end{array}
        &
        \begin{array}{c}
          c_2^4
        \end{array}
        &
        \begin{array}{c}
          \boldsymbol{R}_h^2
        \end{array}
        &
        \begin{array}{c}
          c_1^6
        \end{array}
        &
        \begin{array}{c}
          0
        \end{array}
        &
        \begin{array}{c}
          0
        \end{array}
        &
        \begin{array}{c}
          0
        \end{array}
        \\[1.5mm]
        \hline
        %
        \begin{array}{c}
          0
        \end{array}
        &
        \begin{array}{c}
          0
        \end{array}
        &
        \begin{array}{c}
          a_4^6 \quad a_5^6
        \end{array}
        &
        \begin{array}{c}
           \frac{5}{h}
        \end{array}
        &
        \begin{array}{c}
          a_6^6 \quad a_7^6 \quad a_8^6 \quad 0
        \end{array}
        &
        \begin{array}{c}
          0
        \end{array}
        &
        \begin{array}{c}
          0
        \end{array}
        \\[1.5mm]
        \hline
        %
        \begin{array}{c}
          0
        \end{array}
        &
        \begin{array}{c}
          0
        \end{array}
        &
        \begin{array}{c}
          0
        \end{array}
        &
        \begin{array}{c}
          c_2^6
        \end{array}
        &
        \begin{array}{c}
          \boldsymbol{R}_h^3
        \end{array}
        &
        \begin{array}{c}
          c_1^{10}
        \end{array}
        &
        \begin{array}{c}
          0
        \end{array}
        \\[1.5mm]
        \hline
        %
        \begin{array}{c}
          0
        \end{array}
        &
        \begin{array}{c}
          0
        \end{array}
        &
        \begin{array}{c}
          0
        \end{array}
        &
        \begin{array}{c}
          0
        \end{array}
        &
        \begin{array}{c}
          0 \quad 0 \quad a_8^{10} \quad a_9^{10}
        \end{array}
        &
        \begin{array}{c}
           \frac{3}{h}
        \end{array}
        &
        \begin{array}{c}
          a_{10}^{10}
        \end{array}
        \\[1.5mm]
        \hline
        %
        0 
        & 0 
        & 0 
        & 0 
        & 0 
        & c_2^{10} 
        & R_h 
        \\[1.5mm]
      \end{array}
 \end{pmatrix},\label{eq:matrix_R_net}
    \end{align}
    where $R_h \in \R^{N \times N}$ is the stiffness matrix introduced in \eqref{eq:stiffness_mass_matrices}, and 
    \[
    \boldsymbol{R}_h^1 = \operatorname{Diag}(R_h,R_h,R_h),\quad
    \boldsymbol{R}_h^2 = \operatorname{Diag}(R_h,R_h),\quad
    \boldsymbol{R}_h^3 = \operatorname{Diag}(R_h,R_h,R_h,R_h).
    \]
    In particular, $\boldsymbol{R}_h \in \R^{(10N+3)\times(10N+3)}$, consistently with the dimension of $\Y_h$ in \eqref{eq:notation_vector_sol}. For each $k \in \{4, 6, 10\}$, the row vectors $a_i^k$ collect the couplings between the edge basis functions and the basis function $\phi_0^k$ associated with the corresponding interior vertex, and the column vectors $c_i^k$ are their transposes, so that $\boldsymbol{R}_h$ is symmetric. The symmetry of $\boldsymbol{R}_h$ arises from the symmetry of the product $(\px\phi_j^k, \px\phi_n^i)_{L^2(0,L)}$ for $j, n \in \{1, \dots, N\}$.

    The vector in square brackets above the matrix $\boldsymbol{R}_h$, indicates which sector of the graph each block corresponds to (for example, $\boldsymbol{R}_h^1$ represents the discretization of the dynamics on edges $e_1$, $e_2$, and $e_3$).
    
    Similarly, $\boldsymbol{E}_h$ is a symmetric and positive definite matrix sharing the same structure as \eqref{eq:matrix_R_net}, where the block diagonal matrices are defined using the mass matrix $E_h$ from \eqref{eq:stiffness_mass_matrices}. 
    
    It is worth noting that $\boldsymbol{R}_h$ not only encodes information about the dynamics but also incorporates the topological structure of the graph, including the coupling conditions at the interior vertices.

    \begin{remark}[Matrix point of view]\label{rem:matrix_point_view}
For a fixed metric graph equipped with quasi-uniform meshes, the assembled mass and stiffness matrices satisfy
\[
c_Eh\|\mathbf z\|^2\leq \mathbf z^\top\boldsymbol E_h\mathbf z\leq C_Eh\|\mathbf z\|^2,\qquad \|\boldsymbol R_h\|\leq \frac{C_R}{h},
\]
where the constants may depend on the graph but are independent of $h$. Consequently, for the normalized operator $\widehat{\boldsymbol A}_h=-\boldsymbol E_h^{-1/2}\boldsymbol R_h\boldsymbol E_h^{-1/2}$, one has $\|\widehat{\boldsymbol A}_h\|\leq Ch^{-2}$. If the decomposition $\boldsymbol R_h=\sum_m\boldsymbol R_{h,m}$ satisfies $\mathbf z^\top\boldsymbol R_{h,m}\mathbf z\geq0$ for every $\mathbf z$ and every $m$, and if $\operatorname{Var}[\boldsymbol R_{rb}]\leq C(M)h^{-2}$, then $\operatorname{Var}[\widehat{\boldsymbol A}_{rb}]\leq C(M)h^{-4}$. Hence the RBM estimates derived in Section~\ref{sec:main_results} have the same powers of $h$ and $\delta$ on the graph. The deterministic FEM contribution has the same $L^2$ order under the corresponding edgewise $H^2$ regularity and vertex compatibility conditions, which are satisfied by the manufactured solution considered here.
\end{remark}
    
    The decomposition of $\boldsymbol{R}_h$ will be performed taking into account the structure of the graph. In the next section, we present two examples of decompositions: one with \textit{overlapping} and one \textit{without overlapping} (non-overlapping).

    \subsection{Numerical Simulations}
    
In this section, we illustrate the discretize+RBM method for the heat equation on metric graphs and the associated optimal control problem. We first consider the ten-edge graph introduced above and compare overlapping and non-overlapping decompositions. We then examine the coupled refinement predicted by \Cref{th:convergence_RBM_FE}, test a manufactured solution with nonzero values at the interior vertices, and study the computational scaling on networks with up to one hundred edges. Finally, we illustrate the approximation of the optimal state and control associated with \Cref{th:op_random_approximation}.
     
    \medspace\medspace
    
    \noindent
    \textbf{Computational Setup.} All numerical experiments were executed on a device featuring an AMD Ryzen 9 5900HS processor with Radeon Graphics, clocked at $3.30$ GHz, and equipped with $16.0$ GB of RAM. The experiments were implemented in Python using the NumPy and SciPy libraries. Unless otherwise stated, the time discretization is performed by the implicit Euler method and the diagonal lumped P1 mass matrix, denoted by $\boldsymbol E_h^\ell$, is used.  
    \subsubsection{Overlapping decomposition}\label{sec:overlapping_dec}
        Let us introduce the following decomposition of the matrix $\boldsymbol{R}_h$:
        \begin{align}\label{eq:decomposition_matrix_Ah_by_blocks}
        \boldsymbol{R}_h&=\begin{pmatrix}
                \textcolor{blue1}{\bf B_1} & 0&0\\
                0&0&0\\
                0&0&0
            \end{pmatrix}+\begin{pmatrix}
                0 & 0&0\\
                0&\textcolor{green1}{\bf B_2} &0\\
                0&0&0
            \end{pmatrix}+\begin{pmatrix}
                0&0&0\\
                0&0&0\\
                0&0&\textcolor{red1}{\bf B_3}
            \end{pmatrix}= \boldsymbol{R}_1+ \boldsymbol{R}_2+ \boldsymbol{R}_3,
        \end{align}
        where the matrices $\textcolor{blue1}{\bf B_1}\in \R^{N_1\times N_1}$, $\textcolor{green1}{\bf B_2}\in \R^{N_2\times N_2}$, and $\textcolor{red1}{\bf B_3}\in \R^{N_3\times N_3}$ are given by
        \begin{align}\label{eq:def_matrices_B}
         & \nonumber \textcolor{blue1}{\bf B_1} =\begin{pmatrix}
        \begin{matrix}
        \boldsymbol{R}_{h}^{1} \ \\
        \begin{smallmatrix}
        a_{1}^{4} & a_{2}^{4} & a_{3}^{4}
        \end{smallmatrix}
        \end{matrix} & \begin{matrix}
        c_{1}^{4}\\
        \begin{smallmatrix}
        4/h
        \end{smallmatrix}
        \end{matrix} & \begin{matrix}
        0\\
        \begin{smallmatrix}
        a_{4}^{4}
        \end{smallmatrix}
        \end{matrix}\\
        \begin{matrix}
        0
        \end{matrix} & \begin{matrix}
        c_{2}^{4}
        \end{matrix} & \begin{matrix}
        R_{h,e_4}^{L}
        \end{matrix}
        \end{pmatrix},\,
        \textcolor{green1}{\bf B_2} =\begin{pmatrix}
        R_{h,e_4}^{R} & 0 & c_{1}^{6} & 0 & 0\\
        0 & R_{h} & c_{2}^{6} & 0 & 0\\
        \begin{smallmatrix}
        a_{4}^{6}
        \end{smallmatrix} & \begin{smallmatrix}
        a_{5}^{6}
        \end{smallmatrix} & \begin{smallmatrix}
        5/h
        \end{smallmatrix} & \begin{smallmatrix}
        a_{6}^{6} & a_{7}^{6}
        \end{smallmatrix} & \begin{smallmatrix}
        a_{8}^{6}
        \end{smallmatrix}\\
        0 & 0 & c_{3}^{6} & \boldsymbol{R}_{h}^{2} \  & 0\\
        0 & 0 & 0 & 0 & R_{h,e_8}^{R}
        \end{pmatrix},\\
        &\hspace{3cm}\textcolor{red1}{\bf B_3}=\begin{pmatrix}
        \begin{matrix}
        R_{h,e_8}^{L}
        \end{matrix} & 0 & \begin{matrix}
        c_{1}^{10}
        \end{matrix} & \begin{matrix}
        0
        \end{matrix}\\
        0 & R_{h} & c_{2}^{10} & {\displaystyle 0}\\
        a_{8}^{10} &  a_{9}^{10} & \begin{smallmatrix}
        3/h
        \end{smallmatrix} & \begin{smallmatrix}
        a_{10}^{10}
        \end{smallmatrix}\\
        0 & 0 & c_{3}^{10} & R_{h}
        \end{pmatrix},
    \end{align}
    Here $N_1=4N+1$, $N_2=5N+1$, and $N_3=3N+1$, matching the numbers of degrees of freedom in the corresponding subgraphs. To define the shared-edge blocks, let $K_{e,j}$ be the assembled P1 stiffness contribution of the $j$-th element of edge $e$, let $s_j=j/N$ for $j\in\{0,\ldots,N\}$, and set
\[
\theta(s)=\frac{(1-s)^8}{(1-s)^8+s^8},\quad R_{h,e}^{L}=\sum_{j=0}^{N}\theta(s_j)K_{e,j},\quad R_{h,e}^{R}=\sum_{j=0}^{N}\bigl(1-\theta(s_j)\bigr)K_{e,j}.
\]
Thus $R_{h,e}^{L}+R_{h,e}^{R}=R_{h,e}$, and both matrices are symmetric positive semidefinite. Since $\theta(0)=1$ and $\theta(1)=0$, the endpoint and vertex-coupling entries in \eqref{eq:def_matrices_B} remain assigned to the adjacent vertex blocks. The matrices $\textcolor{blue1}{\bf B_1}$, $\textcolor{green1}{\bf B_2}$, and $\textcolor{red1}{\bf B_3}$ represent the dynamics on the edge groups adjacent to $v_1$, $v_2$, and $v_3$, respectively. In particular, the first group contains $\{e_1,e_2,e_3,e_4\}$, the second contains $\{e_4,e_5,e_6,e_7,e_8\}$, and the third contains $\{e_8,e_9,e_{10}\}$.
        
        To introduce the random system, we consider the subsets $S_i$ of $\{1, 2, 3\}$ (since $M = 3$). Let us decompose $\F_h$ as
        \begin{align*}
            \F_h = (\F_1, 0, 0)^{\top} + (0, \F_2, 0)^{\top} + (0, 0, \F_3)^{\top},
        \end{align*}
        where
        \begin{align*}
    &\F_1=\left(\f_h^{e_1},\f_h^{e_2},\f_h^{e_3},\f_{h,L}^{e_4}\right)^\top,\quad
    \F_2=\left(\f_{h,R}^{e_4},\f_h^{e_5},\f_h^{e_6},\f_h^{e_7},\f_{h,R}^{e_8}\right)^\top,\\
    &\hspace{3cm}\F_3=\left(\f_{h,L}^{e_8},\f_h^{e_9},\f_h^{e_{10}}\right)^\top,
\end{align*}
 where $\f_{h,L}^{e}$ and $\f_{h,R}^{e}$ are assembled from the elemental load vectors using the same weights $\theta(s_j)$ and $1-\theta(s_j)$, respectively. Consequently, $\f_{h,L}^{e}+\f_{h,R}^{e}=\f_h^e$ on each shared edge. We define 
         \begin{align*}
             S_1 = \{1\},\quad S_2 = \{2\},\quad S_3 = \{3\}, \quad \text{and}\quad  p_1 = p_2 = p_3 = 1/3,
         \end{align*}
         and $p_i=0$ for the remaining indices $i\in \{1,\dots 8\}\setminus\{1,2,3\}$. Thus we have that $\pi_1 = \pi_2 = \pi_3 = 1/3$. Now, we can introduce a matrix $\boldsymbol{R}_{rb}$ as in \eqref{eq:definition_matrix_R_rb}. Within this setting, at each switching interval, an index $i\in\{1,2,3\}$ is selected with probability $\pi_i=1/3$, and the random operator is given by $\boldsymbol{R}_{rb}=\boldsymbol{R}_i/\pi_i$ on that interval. Then, define $\Y_R$ as the solution of the equation
        \begin{align}\label{eq:random_network_fem}
            \begin{cases}
                \boldsymbol{E}_h\, \partial_t \Y_R + \boldsymbol{R}_{rb} \Y_R = \F_{rb}, & t \in (0,T), \\
                \boldsymbol{E}_h\Y_R(0)=\Y_h^0.
            \end{cases}
        \end{align}
        Here, following \Cref{remark:mass_matrix}, the function $\F_{rb}$ is defined by
        \begin{align}\label{eq:def_finction_f_rand_net}
            \F_{rb}(t) = \sum_{m \in S_{\omega_k}} \frac{\F_m}{\pi_m}, \quad t \in [(k-1)\delta,\, k\delta),
        \end{align}
        for each $k \in \{1, \dots, K\}$. Similar to $\boldsymbol{R}_{rb}$, the function $\F_{rb}$ also satisfies $\mathbb{E}[\F_{rb}] = \F_h$.
        
For the spatial discretization, we use $N=300$ private interior nodes on each edge, corresponding to $301$ P1 elements per edge and $\mathcal{N}=10N+3=3003$ global degrees of freedom. We take $T=L=1$, $\Delta t=5\times10^{-3}$, and $\delta=10^{-2}=2\Delta t$. The expectation is approximated using $30$ independent realizations. The same mesh, time step, and mass-lumped P1 discretization are used for the Full-Graph FEM and RBM computations. In the manufactured-solution experiments, the initial finite element function is initialized with the nodal P1 interpolant $I_hy^0$.

\begin{remark}[On the time discretization]\label{rem:time_discretization}
The implicit Euler method is used for the Full-Graph and discretize+RBM systems. Its unconditional stability is convenient because the RBM switching interval $\delta$ is already coupled with the spatial mesh size in \Cref{th:convergence_RBM_FE}. In the convergence experiment below, we set $\Delta t=\delta$, so that the spatial mesh size and the switching interval remain the only refinement parameters.
\end{remark}

The errors reported in this experiment are
\[
E_{\mathrm{FEM}}=\max_n\|y_h(t_n)-y(t_n)\|_{L^2(\mathcal G)},\quad E_{\mathrm{RBM}}=\max_n\|\mathbb E_R[y_R(t_n)]-y(t_n)\|_{L^2(\mathcal G)},
\]
where $\mathbb E_R$ denotes the empirical mean over the $30$ realizations. The execution time of the RBM is the mean time required by one realization. The memory measurement is the numerical solver storage retained during time integration, including the compact factors, indexing information, mass data, and load data. Matrix assembly, error evaluation, plotting, and file input/output are excluded from the timing measurements.

The results in \Cref{fig:comparation} show that the empirical mean follows the exact solution on every edge. The corresponding errors and computational measurements are reported in \Cref{tab:memory}.

\begin{figure}[H]
    \centering
    \includegraphics[width=0.9\linewidth]{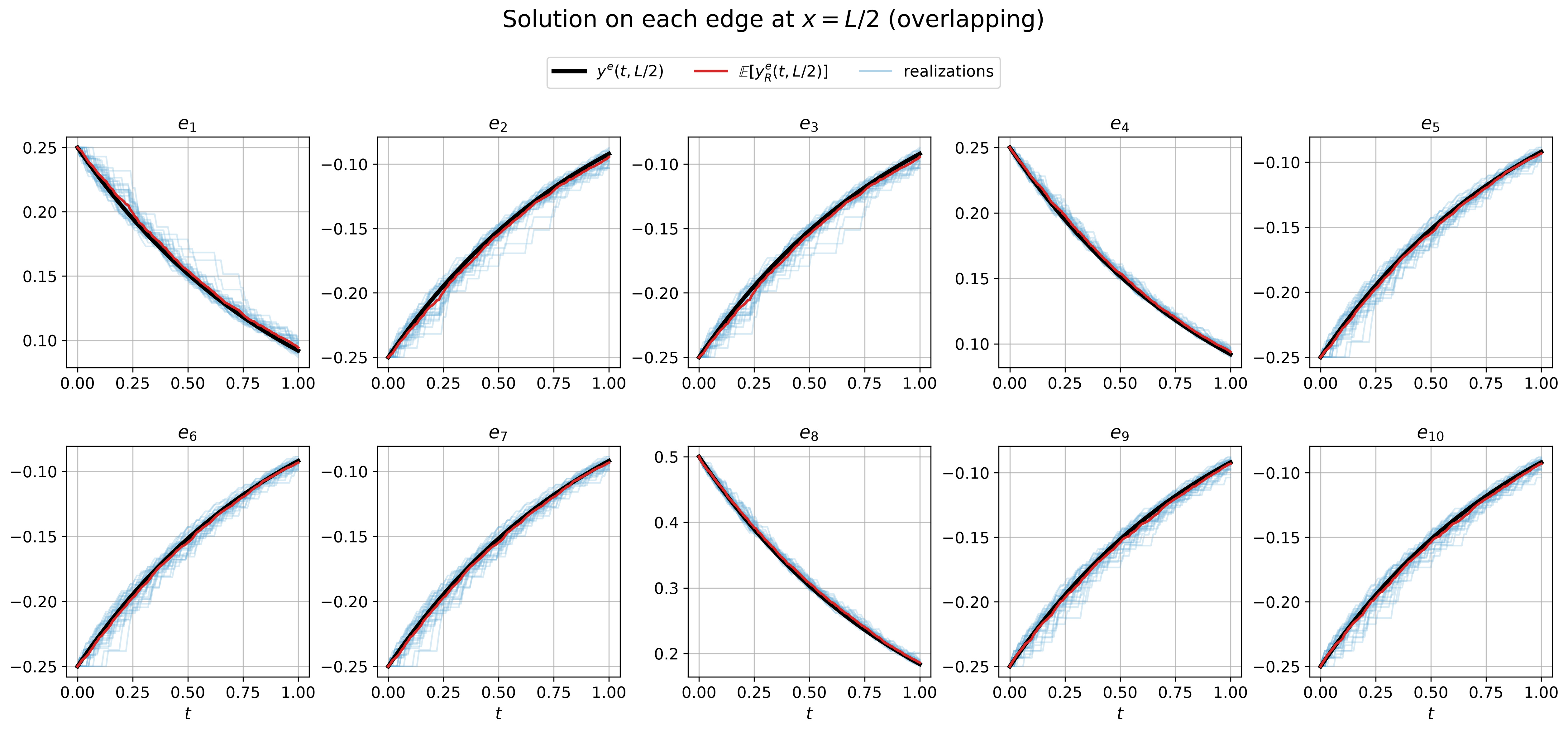}
    \caption{Exact solution $y^e(t,L/2)$, empirical RBM mean $\mathbb E_R[y_R^e(t,L/2)]$, and the $30$ individual realizations on each edge for the overlapping decomposition. Here $N=300$, $\Delta t=5\times10^{-3}$, and $\delta=10^{-2}$.}
    \label{fig:comparation}
\end{figure}
        
\begin{table}[H]
\centering
\caption{Error, mean execution time per trajectory, and retained solver storage for the ten-edge verification problem. The reported uncertainties are sample standard deviations over the independent timing repetitions.}
\label{tab:memory}
\resizebox{\textwidth}{!}{\begin{tabular}{lccc}
\toprule
Method & Execution time (s) & Solver storage (MB) & Error \\
\midrule
Full-Graph FEM & $0.07162\pm0.00393$ & $0.172$ & $1.24\times10^{-4}$ \\
Discretize+RBM, overlapping & $0.04493\pm0.00562$ & $0.156$ & $1.0496\times10^{-2}$ \\
\bottomrule
\end{tabular}}
\end{table}

        Following \Cref{remark:conv_ht_hx}, we next examine the refinement of $h$ and $\delta$ in \eqref{eq:estim_app_heat}. For $\varepsilon>0$, we select $\delta\approx h^{6/(1-\varepsilon)}$ and set $\Delta t=\delta$. The value of $\delta$ is projected onto a temporal grid that reaches $T=1$ exactly. In the computations below, $\varepsilon=10^{-4}$ and $30$ realizations are used.

We record two RMS quantities:
\begin{align}\label{eq:total_empirical_error}
E_{\mathrm{tot}}=\max_n\left(\frac1{30}\sum_{r=1}^{30}\|y_R^{(r)}(t_n)-y(t_n)\|_{L^2(\mathcal G)}^2\right)^{1/2},
\end{align}
and
\begin{align}\label{eq:error_galerkin_empirical}
E_{\mathrm{batch}}=\max_n\left(\frac1{30}\sum_{r=1}^{30}\|y_R^{(r)}(t_n)-y_h(t_n)\|_{L^2(\mathcal G)}^2\right)^{1/2}.
\end{align}
Error \eqref{eq:total_empirical_error} quantifies the total empirical error between the RBM and exact solutions. The second error \eqref{eq:error_galerkin_empirical} isolates the random-batch discrepancy from the spatial finite element error. As illustrated in \Cref{fig1}(a), a log--log fit with respect to $\delta$ gives the empirical rates $E_{\mathrm{tot}}=O(\delta^{0.334})$ and $E_{\mathrm{batch}}=O(\delta^{0.508})$, with coefficients of determination $R^2=1.000$ and $R^2=0.977$, respectively. Since $\delta\approx h^{6/(1-\varepsilon)}$, the deterministic FEM rate $O(h^2)$ corresponds to $O(\delta^{(1-\varepsilon)/3})$, which agrees with the observed slope $0.334$ of $E_{\mathrm{tot}}$. The fitted exponent $0.508$ for $E_{\mathrm{batch}}$ is consistent with the $O(\delta^{1/2})$ root-mean-square dependence given by the finite-dimensional RBM theory in \cite{MR4433122}.

The manufactured spatial profiles provide a direct consistency check for the spatial discretization. Since the interior-vertex values vanish in this example, the profile on each edge is $\phi_e(x)=p_e x(1-x)$. On a uniform mesh, its P1 interpolation error satisfies $\|\phi_e-I_h\phi_e\|_{L^2(0,L)}^2=\frac{p_e^2}{30}h^4.$ Using $\sum_{e=1}^{10}p_e^2=13$, we obtain $\|y(0)-I_hy(0)\|_{L^2(\mathcal G)}=\sqrt{\frac{13}{30}}\,h^2.$ In the present experiment, this initial interpolation error is also the maximum Full-Graph FEM error over the temporal grid. It therefore gives
\[
E_{\mathrm{FEM}}(h)=\max_n\|y_h(t_n)-y(t_n)\|_{L^2(\mathcal G)}=\sqrt{\frac{13}{30}}\,h^2,
\]
as reflected in \Cref{tab:ab}. The quantity $E_{\mathrm{batch}}$ reported in the same table separately measures the discrepancy introduced by the random batching.

Figure~\ref{fig1}(b) presents a work--precision comparison on the fixed mesh with $N=300$ private interior nodes per edge. For Full-Graph FEM, only $\Delta t=T/K$ is varied. For RBM, $\Delta t=\delta=T/K$, and every point is obtained from $30$ realizations. Both methods are compared using the $L^2(0,T;L^2(\mathcal G))$ trajectory error. At a fixed value of $K$, the RBM trajectory is between $1.51$ and $1.84$ times faster, while its error also contains the random-batch contribution. The overlapping error ranges of the two curves provide a direct comparison at prescribed accuracy on this ten-edge graph. Since the spatial mesh and the three batch supports remain fixed, varying $\Delta t$ and $\delta$ changes the execution time and error but not the retained solver storage, which is $0.172$ MB for Full-Graph FEM and $0.156$ MB for RBM.

\begin{figure}[H]
    \centering
    \begin{subfigure}{0.48\linewidth}
        \centering
        \includegraphics[width=0.9\linewidth]{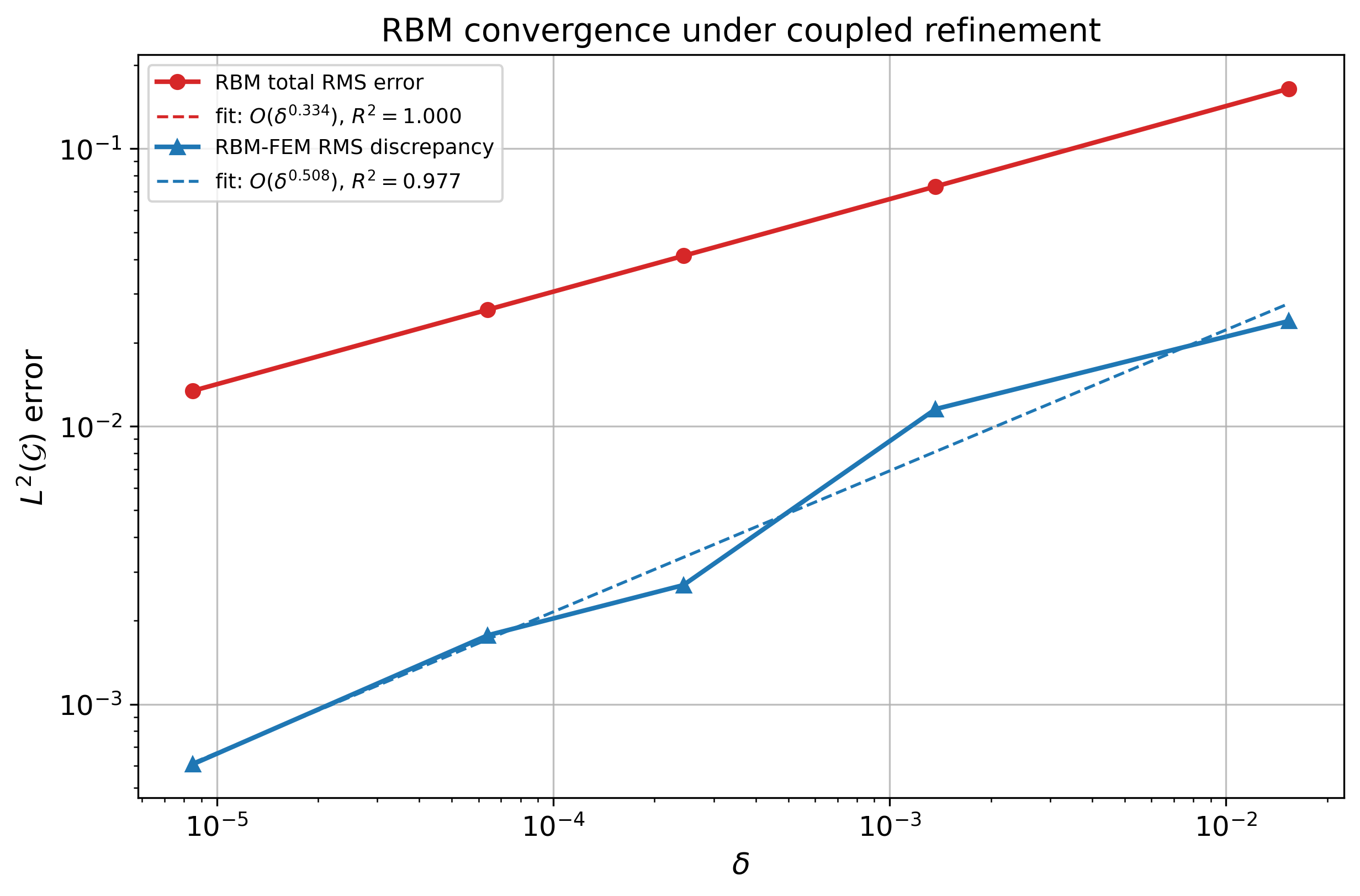}
        \caption{}
    \end{subfigure}
    \hfill
    \begin{subfigure}{0.48\linewidth}
        \centering
        \includegraphics[width=0.9\linewidth]{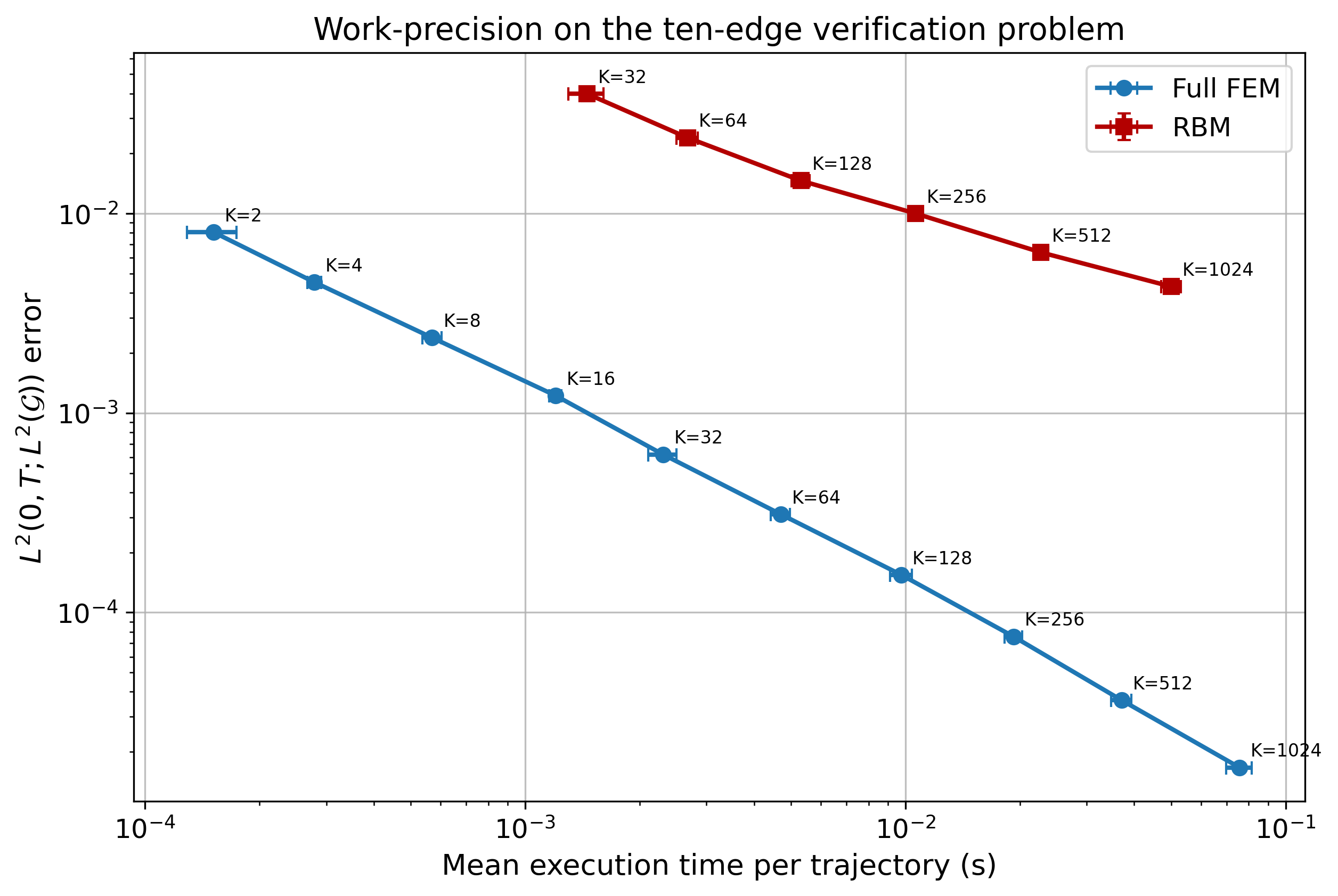}
        \caption{}
    \end{subfigure}
    \caption{(a) Total RBM RMS error and RBM--FEM RMS discrepancy under the coupled refinement $\delta\approx h^{6/(1-\varepsilon)}$, with $\Delta t=\delta$, $\varepsilon=10^{-4}$, and $30$ realizations. The fitted slopes are computed with respect to the displayed variable $\delta$. (b) Work--precision diagram for the ten-edge verification problem with $3003$ global degrees of freedom. The horizontal coordinate is the mean execution time per trajectory; $K=T/\Delta t$, and $\delta=\Delta t$ for RBM. The retained solver storage is $0.172$ MB for Full-Graph FEM and $0.156$ MB for RBM.}
    \label{fig1}
\end{figure}
\begin{table}[H]
\centering
\caption{Coupled FEM/RBM refinement. Here $\Delta t=\delta$, $E_{\mathrm{tot}}$ is the total RBM RMS error, $E_{\mathrm{batch}}$ is the RBM--FEM RMS discrepancy, and $E_{\mathrm{FEM}}=\sqrt{13/30}\,h^2$.}
\label{tab:ab}
\resizebox{0.9\textwidth}{!}{
\begin{tabular}{ccccc}
\toprule
$h$ & $\delta$ & $E_{\mathrm{tot}}$ & $E_{\mathrm{batch}}$ & $E_{\mathrm{FEM}}$ \\
\midrule
$5.0000\times10^{-1}$ & $1.5385\times10^{-2}$ & $1.6457\times10^{-1}$ & $2.3980\times10^{-2}$ & $1.6457\times10^{-1}$ \\
$3.3333\times10^{-1}$ & $1.3699\times10^{-3}$ & $7.3142\times10^{-2}$ & $1.1551\times10^{-2}$ & $7.3142\times10^{-2}$ \\
$2.5000\times10^{-1}$ & $2.4390\times10^{-4}$ & $4.1143\times10^{-2}$ & $2.6842\times10^{-3}$ & $4.1143\times10^{-2}$ \\
$2.0000\times10^{-1}$ & $6.3935\times10^{-5}$ & $2.6331\times10^{-2}$ & $1.7707\times10^{-3}$ & $2.6331\times10^{-2}$ \\
$1.4286\times10^{-1}$ & $8.4899\times10^{-6}$ & $1.3434\times10^{-2}$ & $6.0797\times10^{-4}$ & $1.3434\times10^{-2}$ \\
\bottomrule
\end{tabular}}
\end{table}

    \subsubsection{Nonoverlapping Decomposition}
    
    In this example, we consider the same decomposition of the matrix $\boldsymbol{R}_h$ introduced in \eqref{eq:decomposition_matrix_Ah_by_blocks}, but with the following block matrices:
    
    \begin{align}\label{B_nonoverlapping}
      & \nonumber\textcolor{blue1}{\bf B_1} =\begin{pmatrix}
    \begin{matrix}
    \boldsymbol{R}_{h}^{1} \ \\
    \begin{smallmatrix}
    a_{1}^{4} & a_{2}^{4} & a_{3}^{4}
    \end{smallmatrix}
    \end{matrix} & \begin{matrix}
    c_{1}^{4}\\
    \begin{smallmatrix}
    4/h
    \end{smallmatrix}
    \end{matrix} & \begin{matrix}
    0\\
    \overline{a_{4}^{4}} 
    \end{matrix}\\
    \begin{matrix}
    0
    \end{matrix} & \begin{matrix}
    c_{2}^{4}
    \end{matrix} & \begin{matrix}
    \overline{R_h}
    \end{matrix}
    \end{pmatrix},\quad \textcolor{green1}{\bf B_2} =\begin{pmatrix}
    \underline{R_h} & 0 & c_{1}^{6} & 0 & 0\\
    0 & R_{h} & c_{2}^{6} & 0 & 0\\
    \begin{smallmatrix}
    \underline{a_{4}^{6}}
    \end{smallmatrix} & \begin{smallmatrix}
    a_{5}^{6}
    \end{smallmatrix} & \begin{smallmatrix}
    5/h
    \end{smallmatrix} & \begin{smallmatrix}
    a_{6}^{6} & a_{7}^{6}
    \end{smallmatrix} & \begin{smallmatrix}
    \overline{a_{8}^{6}}
    \end{smallmatrix}\\
    0 & 0 & c_{3}^{6} & \boldsymbol{R}_{h}^{2} \  & 0\\
    0 & 0 & 0 & 0 & \overline{R_h}
    \end{pmatrix},\\[2mm]
    &\hspace{3cm}\textcolor{red1}{\bf B_3}=\begin{pmatrix}
    \underline{R_h} & 0 & c_{1}^{10} & 0\\
    0 & R_{h} & c_{2}^{10} & 0\\
    \underline{a_{8}^{10}} & a_{9}^{10} & 3/h & a_{10}^{10}\\
    0 & 0 & c_{3}^{10} & R_{h}
    \end{pmatrix},
    \end{align}
    where the matrix $\overline{R_h}$ is given by
    \begin{align}\label{eq:definition_over_under_R}
                \overline{R_h} = \begin{pmatrix}
         R_{1}^* & \biggr.\biggr|\begin{matrix} 0\\ -1/h \end{matrix} \\ 
    \overline{\begin{matrix} 0& -1/h \end{matrix}} & 1/h 
        \end{pmatrix} \in \R^{\overline{N}\times \overline{N}},\quad  \underline{R_h} = \begin{pmatrix}
         1/h &  \underline{\begin{matrix} -1/h& 0 \end{matrix}} \\ \begin{matrix} -1/h\\ 0 \end{matrix}\biggr.\biggr|  & R_{2}^* 
        \end{pmatrix} \in \R^{\underline{N}\times \underline{N}},
    \end{align}
    with $\overline{N}=\lfloor N/2\rfloor-1$, $\underline{N}=N-\overline{N}+1$, and $R_1^*$ and $R_2^*$ being the square matrices of the form $R_h$ introduced in \eqref{eq:stiffness_mass_matrices}, with sizes $(\overline{N}-1)\times(\overline{N}-1)$ and $(\underline{N}-1)\times(\underline{N}-1)$, respectively. Here, $\underline{a_i^k}$ is a row vector of dimension $\underline N$ and $\overline{a_i^k}$ is a row vector of dimension $\overline N$.

    We observe that the matrix $\overline{R_h}$ encodes the dynamics on half of an edge. In the present example, the block matrix $\textcolor{blue1}{\bf B_1}$ represents the dynamics on edges $e_1$, $e_2$, $e_3$, and the first half of edge $e_4$, while $\textcolor{green1}{\bf B_2}$ corresponds to the dynamics on edges $e_5$, $e_6$, $e_7$, the second half of $e_4$, and the first half of $e_8$. Finally, $\textcolor{red1}{\bf B_3}$ captures the dynamics on edges $e_9$, $e_{10}$, and the second half of $e_8$ (see \Cref{fig:non_overlapping_illustration}). We refer to this as a \textit{non-overlapping decomposition} because the block matrices share dynamics only at a single discretization point. In contrast, the \textit{overlapping decomposition} (presented in \Cref{sec:overlapping_dec}) features an overlap of $N$ nodes (since the block matrices share a matrix of $N\times N$), so that the magnitude of shared information increases with $N$, whereas in the non-overlapping case, the shared information remains constant regardless of $N$.

    \begin{figure}[h!]
        \centering
        \includegraphics[scale=0.25]{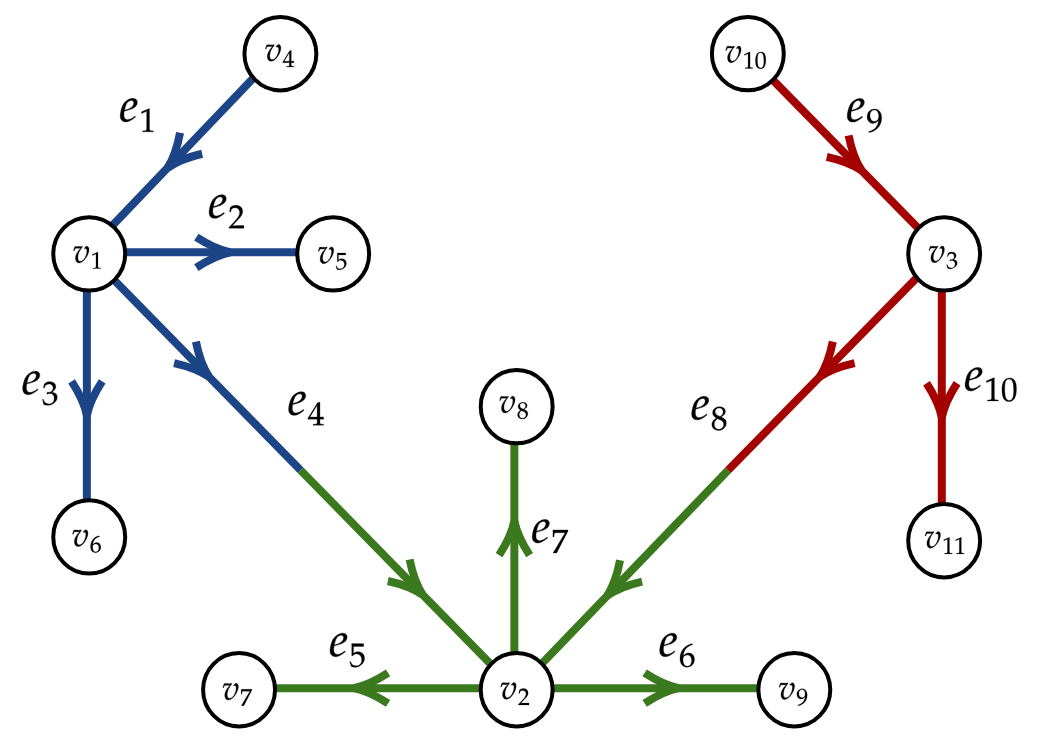}
        \caption{Illustration of the information contained in each block matrix. We illustrate in blue the information of $\textcolor{blue1}{\bf B_1}$, in green the information of $\textcolor{green1}{\bf B_2}$, and in red the information of $\textcolor{red1}{\bf B_3}$.}
        \label{fig:non_overlapping_illustration}
    \end{figure}
    
    Then, similar to \eqref{eq:decomposition_matrix_Ah_by_blocks}, we introduce a decomposition of the matrix $\boldsymbol{R}_h = \boldsymbol{R}_1 + \boldsymbol{R}_2 + \boldsymbol{R}_3$ using the blocks $\bf{B_i}$ introduced in \eqref{B_nonoverlapping}. Furthermore, we decompose $\F_h$ as $  \F_h = (\F_1, 0, 0)^{\top} + (0, \F_2, 0)^{\top} + (0, 0, \F_3)^{\top},$ where
    
    \begin{align}\label{eq:decom_f_nonooverlap}
        \begin{aligned}
            &\F_1 = \left( \f_h^{e_1}, \f_h^{e_2}, \f_h^{e_3}, [\f_h^{e_4}]_{I_1} \right)^{\top}, \,
        \F_2 = \left( [\f_h^{e_4}]_{I_2}, \f_h^{e_5}, \f_h^{e_6}, \f_h^{e_7}, [\f_h^{e_8}]_{I_1} \right)^{\top}, \\
        &\hspace{3cm}\F_3 = \left( [\f_h^{e_8}]_{I_2}, \f_h^{e_9}, \f_h^{e_{10}} \right)^{\top}.
        \end{aligned}
    \end{align}
    
    Here, $I_1 = \{0, \dots, \overline{N}\}$ and $I_2 = \{\overline{N}, \dots, N\}$. Using \eqref{eq:decom_f_nonooverlap}, we define $\F_{rb}$ according to \eqref{eq:def_finction_f_rand_net}, and $\boldsymbol{R}_{rb}$ as in \eqref{eq:definition_matrix_R_rb}. Then, we introduce the system \eqref{eq:random_network_fem}.
    
    For this experiment, we use the same spatial mesh, time discretization, switching interval, and $30$ realizations as in the overlapping case. Figure~\ref{fig:non_overlapping_comparation} shows the exact solution, the empirical RBM mean, and the individual realizations at the midpoint of each edge. The corresponding computational results are summarized in \Cref{tab:memory_non}.

\begin{figure}[H]
    \centering
    \includegraphics[width=0.9\linewidth]{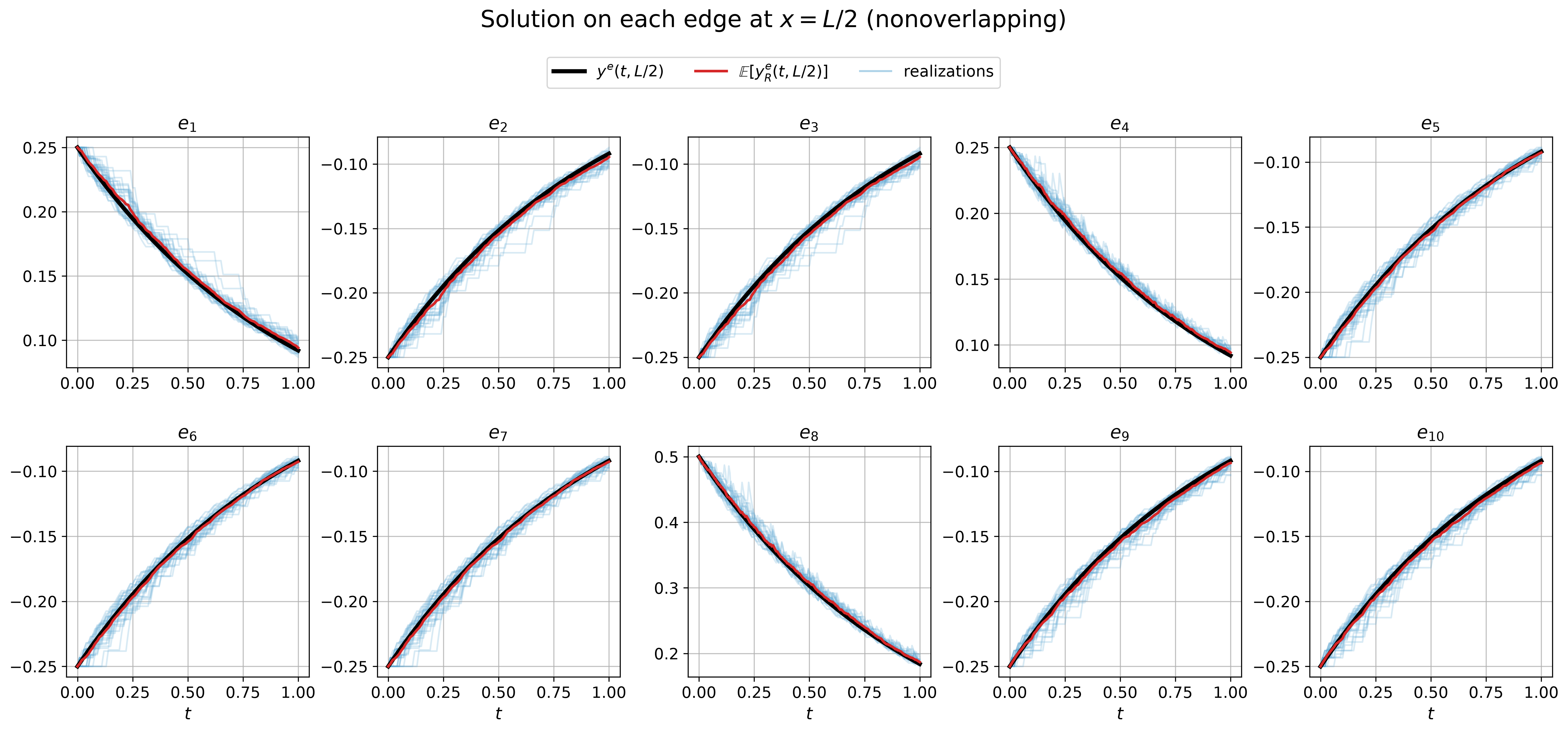}
    \caption{Exact solution $y^e(t,L/2)$, empirical RBM mean $\mathbb E_R[y_R^e(t,L/2)]$, and the $30$ individual realizations on each edge for the non-overlapping decomposition. Here $N=300$, $\Delta t=5\times10^{-3}$, and $\delta=10^{-2}$.}
    \label{fig:non_overlapping_comparation}
\end{figure}

\begin{table}[H]
\centering
\caption{Error, mean execution time per trajectory, and retained solver storage for the non-overlapping decomposition.}
\label{tab:memory_non}
\resizebox{\textwidth}{!}{\begin{tabular}{lccc}
\toprule
Method & Execution time (s) & Solver storage (MB) & Error \\
\midrule
Full-Graph FEM & $0.07162\pm0.00393$ & $0.172$ & $1.24\times10^{-4}$ \\
Discretize+RBM, non-overlapping & $0.04027\pm0.00602$ & $0.130$ & $1.0140\times10^{-2}$ \\
\bottomrule
\end{tabular}}
\end{table}

     \begin{remark}[Computational Efficiency of RBM]\label{remark:computational_efficiency}
To explain the reduction in computational cost, first consider the consistent P1 mass matrix $\boldsymbol{E}_h$, whose sparsity pattern is compatible with that of the stiffness matrix $\boldsymbol{R}_h$. After ordering the degrees of freedom according to the three subgraphs, when $\boldsymbol{R}_1$ is selected on a switching interval $(t_k,t_{k+1})$, the random system takes the block form
    \begin{align}\label{eq:system_one_realization}
    \begin{pmatrix}
    E_{11}^1 & E_{12}^1 & 0 \\
    E_{21}^1 & E_{22}^1 & E_{23}^1 \\
    0 & E_{32}^1 & E_{33}^1
    \end{pmatrix}
    \partial_t
    \begin{pmatrix}
    \boldsymbol{Y}_R^1 \\ \boldsymbol{Y}_R^2 \\ \boldsymbol{Y}_R^3
    \end{pmatrix}
    +
    \begin{pmatrix}
    \textcolor{blue1}{\mathbf{B}_1}/\pi_1 & 0 & 0 \\
    0 & 0 & 0 \\
    0 & 0 & 0
    \end{pmatrix}
    \begin{pmatrix}
    \boldsymbol{Y}_R^1 \\ \boldsymbol{Y}_R^2 \\ \boldsymbol{Y}_R^3
    \end{pmatrix}
    =
    \begin{pmatrix}
    \boldsymbol{F}_1/\pi_1 \\ 0 \\ 0
    \end{pmatrix}.
    \end{align}
    
    Due to the block structure of this system, the only differential equation that needs to be solved is the reduced system
    \begin{align}\label{eq:equation_reduced}
    \pi_1 \left(
    E_{11} - E_{12} \left( E_{22} - E_{23} E_{33}^{-1} E_{32} \right)^{-1} E_{21}
    \right) \partial_t \boldsymbol{Y}_R^1 + \textcolor{blue1}{\mathbf{B}_1} \boldsymbol{Y}_R^1 = \boldsymbol{F}_1.
    \end{align}
 Since $\boldsymbol{E}_h$ is symmetric positive definite, both $E_{33}$ and the Schur complement $E_{22}-E_{23}E_{33}^{-1}E_{32}$ are symmetric positive definite. Hence \eqref{eq:equation_reduced} is a reduced differential system whose dimension is strictly smaller than the full dimension $\mathcal N=10N+3$. Once $\boldsymbol{Y}_R^1$ has been computed, the remaining components are recovered from
    \begin{align}\label{eq:relationfor_y2}
    \boldsymbol{Y}_R^2(t) = \boldsymbol{Y}_R^2(t_k) - 
    \left( E_{22} - E_{23} E_{33}^{-1} E_{32} \right)^{-1} E_{21} \big( \boldsymbol{Y}_R^1(t) - \boldsymbol{Y}_R^1(t_k) \big),
    \end{align}
    and
    \begin{align}\label{eq:relationfor_y3}
    \boldsymbol{Y}_R^3(t) = \boldsymbol{Y}_R^3(t_k) - E_{33}^{-1} E_{32} \big( \boldsymbol{Y}_R^2(t) - \boldsymbol{Y}_R^2(t_k) \big).
    \end{align}
The inverse notation in \eqref{eq:equation_reduced}--\eqref{eq:relationfor_y3} represents the solution of linear systems. The corresponding factorizations can be computed once and reused throughout all time steps associated with the same batch. Therefore, recovering $\boldsymbol{Y}_R^2$ and $\boldsymbol{Y}_R^3$ requires no additional matrix factorizations.
    
Consequently, the time-dependent linear solve is restricted to the active batch, while the remaining components are recovered through previously factorized mass-matrix blocks. For the lumped mass matrix used in the numerical experiments, the off-diagonal mass blocks vanish and this reduction becomes a direct solve on the active degrees of freedom. This structure accounts for the reductions in solver storage and computational work observed in the numerical experiments.
    \end{remark}

    \begin{remark}[Non‐overlapping Decomposition]\label{remark:nonoverlapping}
    
    We refer to the previous method as a non-overlapping decomposition due to its representation on the graph. However, in classical non-overlapping domain decomposition, such as the Dirichlet-Neumann method, one imposes transmission conditions at interfaces to enforce flux continuity between adjacent subdomains. By contrast, the RBM splits the global stiffness matrix $\boldsymbol{R}_h$ into submatrices as in \eqref{eq:decom_A}, without introducing explicit interface operators.  
    
    For the consistent P1 mass matrix, information propagates across the subgraphs through the off-diagonal mass-matrix couplings.  Indeed, in the non‐overlapping setting, in \eqref{eq:system_one_realization} when $\mathbf{R}_1$ is chosen on the time interval $[\,t_k,t_{k+1}\,]$, the only nonzero off‐diagonal blocks of the mass matrix are rank‐one. Concretely, in \eqref{eq:system_one_realization} one finds
    \begin{equation}\label{eq:submatrices_mass_revised}
        E_{12} =\frac{h}{6}\mathbf{e}_{N_1}\mathbf{e}_{1}^\top,\quad 
        E_{23} =\frac{h}{6}\mathbf{e}_{N_2}\mathbf{e}_{1}^\top,\quad 
        E_{21}=\frac{h}{6}\mathbf{e}_{1}\mathbf{e}_{N_1}^\top,\quad  
        E_{32}=\frac{h}{6}\mathbf{e}_{1}\mathbf{e}_{N_2}^\top,
    \end{equation}
    where $\mathbf{e}_{i}$ is the $i$-th canonical basis vector, and $N_{1}=3N+\overline N+1$, $
    N_{2}=3N+\underline N+\overline N$ and $
    N_{3}=2N+\underline N$. Observe that the information spreads along the whole graph because of \eqref{eq:relationfor_y2} and \eqref{eq:relationfor_y3}. Moreover, in this case, the coupling into $\boldsymbol{Y}_R^2$ and $\boldsymbol{Y}_R^3$ is governed by $\mathbf{Y}_{R,N_1}^1$, the  $N_1$-th component of $\mathbf{Y}_R^1$. Namely,
    \begin{align}\label{eq:non_overlapping_coupling1}
    \boldsymbol{Y}_R^2(t)&=\boldsymbol{Y}_R^2(t_k)-\left(E_{22}-E_{23}E_{33}^{-1}E_{32}\right)^{-1}E_{21}\left(\boldsymbol{Y}_R^1(t)-\boldsymbol{Y}_R^1(t_k)\right).
    \\
    \nonumber&= \boldsymbol{Y}_R^2(t_k)\\
    &\quad\nonumber-\frac{h}{6}\left(\boldsymbol{Y}_{R,N_1}^1(t) -\boldsymbol{Y}_{R,N_1}^1(t_k)\right)\left[\left( E_{22} - \frac{h^2}{36}\mathbf{e}_{N_2}\bigl(\mathbf{e}_{1}^\top E_{33}^{-1}\mathbf{e}_{1}\bigr)\mathbf{e}_{N_2}^\top\right)^{-1}\mathbf{e}_{1}\right] \\
    &=: \boldsymbol{Y}_R^2(t_k)
    -\frac{h}{6}\left(\boldsymbol{Y}_{R,N_1}^1(t) -\boldsymbol{Y}_{R,N_1}^1(t_k)\right)\bigl[\mathbf{T}^{-1}\mathbf{e}_{1}\bigr].
    \end{align}
    Similarly, 
    \begin{align}\label{eq:non_overlapping_coupling2}
    \nonumber\boldsymbol{Y}_R^3(t)
    &= \boldsymbol{Y}_R^3(t_k) - E_{33}^{-1}E_{32}\bigl(\boldsymbol{Y}_R^2(t)-\boldsymbol{Y}_R^2(t_k)\bigr)
    \\
    &= \boldsymbol{Y}_R^3(t_k)
    +\frac{h^2}{36}\left(\boldsymbol{Y}_{R,N_1}^1(t) -\boldsymbol{Y}_{R,N_1}^1(t_k)\right)
    \bigl[\mathbf{T}^{-1}\mathbf{e}_{1}\bigr]_{N_2}\,
    \bigl[E_{33}^{-1}\mathbf{e}_{1}\bigr],
    \end{align}
    where $\bigl[\mathbf{T}^{-1}\mathbf{e}_{1}\bigr]_{N_2}$ denotes the $N_2$‐th component of $\mathbf{T}^{-1}\mathbf{e}_{1}$. Consequently, even when only $\mathbf{R}_1$ drives the evolution of \eqref{eq:system_one_realization}, the rank‐one entries in $E_h$ ensure that the other components are \emph{never} decoupled: each of the vectors $\boldsymbol{Y}_R^2$ and $\boldsymbol{Y}_R^3$ is {influenced} by the boundary value $\boldsymbol{Y}_{R,N_1}^1$. In this sense, the information is propagated through the entire graph. Furthermore, the nonzero interface blocks satisfy $E_{ij}=O(h)$, while the Euclidean operator norms of $\mathbf{T}^{-1}$ and $E_{33}^{-1}$ are $O(h^{-1})$. Consequently, the transfer operators appearing in \eqref{eq:non_overlapping_coupling1} and \eqref{eq:non_overlapping_coupling2} remain of order one under mesh refinement.
    \end{remark}

    \begin{remark}[Comments on finite differences]
    The discretization using finite differences within each edge is analogous to the interval case. The discretization of the Neumann coupling conditions can be done by considering a first-order approximation using the boundary node and the previous one (inside the edge). 
    Therefore, the discretization of the coupling conditions in \eqref{eq:heat_bc} defines an algebraic equation, and thus, the discretization of \eqref{eq:heat_net}-\eqref{eq:heat_bc} defines a differential-algebraic equation (DAE). Observe that the discretization at node $v_1$ leads to
    \begin{align}\label{eq:discrete_coupling_cond}
          4y_{(0)}^{e_4} = y^{e_1}_{(N)} + y^{e_2}_{(1)} + y^{e_3}_{(1)} + y^{e_4}_{(1)}.
    \end{align}
    At this point, the discretization does not align with the RBM presented in \Cref{sec:rbm}. However, by differentiating \eqref{eq:discrete_coupling_cond} with respect to time, we obtain the following differential equation representing the coupling condition:
    \begin{align}\label{eq:new_diff_coupling}
\pt y_{(0)}^{e_4}
&=\frac{1}{4h^2}\Bigl(4y_{(0)}^{e_4}-2\bigl(y_{(N)}^{e_1}+y_{(1)}^{e_2}+y_{(1)}^{e_3}+y_{(1)}^{e_4}\bigr)+\bigl(y_{(N-1)}^{e_1}+y_{(2)}^{e_2}+y_{(2)}^{e_3}+y_{(2)}^{e_4}\bigr)\Bigr)\nonumber\\
&\quad+\frac14\bigl(f_{(N)}^{e_1}+f_{(1)}^{e_2}+f_{(1)}^{e_3}+f_{(1)}^{e_4}\bigr).
\end{align}
    Assume that $y^0\in H_0^1(\mathcal E)\cap H^2(\mathcal E)$ satisfies the Kirchhoff condition in \eqref{eq:heat_bc}. Then condition \eqref{eq:new_diff_coupling} is equivalent to \eqref{eq:discrete_coupling_cond}. Applying this to all interior vertices, we can express the discretization of \eqref{eq:heat_net}-\eqref{eq:heat_bc} as a single ODE. 
    \end{remark}

    \subsubsection{A Manufactured Solution with Nonzero Interior-Vertex Values}\label{sec:nonzero_vertex_example}

We next consider a manufactured solution that exercises the continuity and Kirchhoff conditions with nonzero values at every interior vertex. Let $\gamma:\mathcal V\to\mathbb R$ be a prescribed vertex function with $\gamma_v=0$ for every $v\in\mathcal V_b$, so that the Dirichlet condition in \eqref{eq:heat_bc} holds. For each edge $e\in\mathcal E$, denote by $a_e$ and $b_e$ its initial and terminal vertices, identified with $x=0$ and $x=L=1$, respectively, and define
\begin{equation}\label{eq:exact_sol_non_null_vertex}
    \phi_e(x)=(1-x)\gamma_{a_e}+x\gamma_{b_e}+p_e x(1-x),\qquad y^e(t,x)=e^{-t}\phi_e(x),
\end{equation}
where $p_e\in\mathbb R$ are coefficients to be determined, and $f^e(t,x)=e^{-t}\bigl(2p_e-\phi_e(x)\bigr).$ Then $\partial_t y^e-\partial_{xx}y^e=f^e$ on every edge, and $y^e(t,v)=e^{-t}\gamma_v$ at each vertex $v$ of $e$, so continuity at the interior vertices holds by construction. We prescribe $\gamma_{v_1}=1$, $\gamma_{v_2}=2$, and $\gamma_{v_3}=1$. With the edge orientation introduced above, the coefficients are
\[
(p_{e_1},\ldots,p_{e_{10}})=\left(\frac53,-\frac13,-\frac13,1,\frac53,\frac53,\frac53,2,-\frac12,-\frac12\right).
\]
These coefficients satisfy $\sum_{e\in\mathcal E(v)}n_e(v)\,\partial_x\phi_e(v)=0$ for every $v\in\mathcal V_0$. Hence continuity and the Kirchhoff condition are both satisfied nontrivially.

We use $N=300$, $T=1$, and $\Delta t=\delta=5\times10^{-3}$. The RBM statistics are computed from $30$ realizations using the overlapping decomposition. Figure~\ref{fig:nonzero_vertices} shows that the empirical mean follows the exact solution on all ten edges, including the two edges joining pairs of interior vertices.

\begin{figure}[H]
    \centering
    \includegraphics[width=0.9\linewidth]{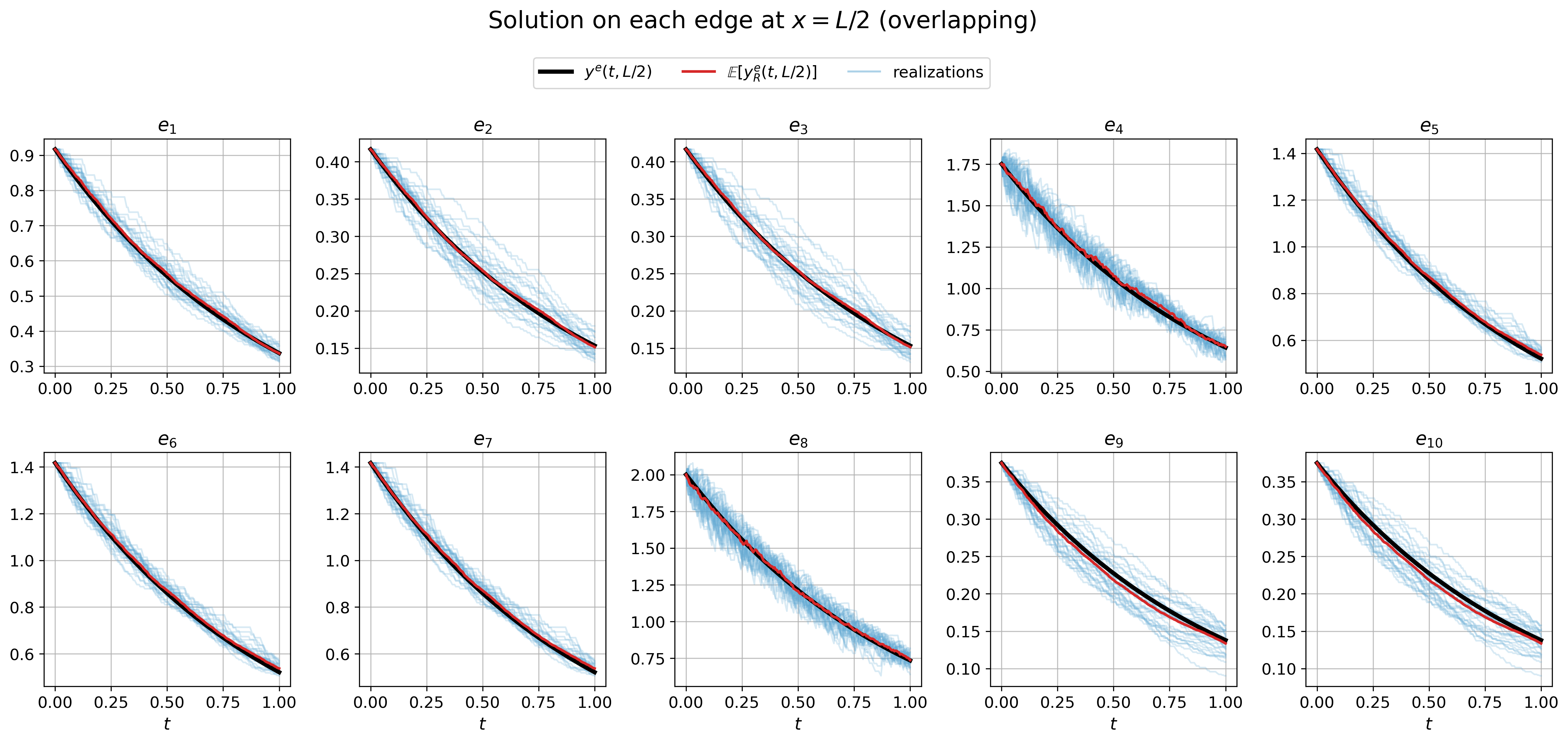}
    \caption{Exact solution, empirical RBM mean, and individual RBM realizations at $x=L/2$ for the manufactured solution with $\gamma_{v_1}=1$, $\gamma_{v_2}=2$, and $\gamma_{v_3}=1$.}
    \label{fig:nonzero_vertices}
\end{figure}

\begin{table}[H]
\centering
\caption{Accuracy and computational measurements for the manufactured solution with nonzero interior-vertex values. The mean error is $\max_n\|\mathbb E_R[y_R(t_n)]-y(t_n)\|_{L^2(\mathcal G)}$, and the RMS error is the corresponding root-mean-square error over the individual realizations.}
\label{tab:nonzero_vertices}
\resizebox{\textwidth}{!}{\begin{tabular}{lcccc}
\toprule
Method & Execution time (s) & Solver storage (MB) & Mean error & RMS error \\
\midrule
Full-Graph FEM & $0.01539\pm0.00099$ & $0.172$ & $2.6826\times10^{-3}$ & $2.6826\times10^{-3}$ \\
Discretize+RBM & $0.00856\pm0.00099$ & $0.156$ & $4.9998\times10^{-2}$ & $1.3316\times10^{-1}$ \\
\bottomrule
\end{tabular}}
\end{table}
    
\subsubsection{Scaling with the Network Size}\label{sec:large_network_experiment}

To examine the dependence on the network size, we consider deterministic nearly balanced binary metric trees with $20$, $50$, and $100$ edges. The corresponding numbers of batches are $M=5$, $M=10$, and $M=20$. Each batch is associated with a connected group of interior vertices, and the number of edges represented by one batch remains approximately constant as the graph grows.

Every edge has length one and is discretized using $100$ P1 elements. We set $T=1$ and $\Delta t=\delta=5\times10^{-4}$. The Full-Graph timing is averaged over five independent repetitions, whereas the RBM timing and accuracy are computed from $30$ realizations. The execution time is the mean time per trajectory.

For an edge $e=(a_e,b_e)$, the manufactured solution $y^e$ is defined as in \eqref{eq:exact_sol_non_null_vertex}. Let $(\xi_v,\eta_v)\in[0,1]^2$ denote the normalized layered position of an interior vertex $v$, let $d_v$ be its depth in the rooted tree, and let $d_{\max}$ be the maximum depth. We set $\gamma_v=1+0.18\sin(2\pi\eta_v)+0.12\frac{d_v}{d_{\max}}+0.05\xi_v$ at the interior vertices and $\gamma_v=0$ at the boundary vertices. The coefficients $(p_e)_{e\in\mathcal E}$ are chosen as the minimum-Euclidean-norm solution of the full-row-rank linear system imposing the Kirchhoff condition at every interior vertex. Thus the manufactured solution is continuous and satisfies a nontrivial flux balance throughout the network.

\begin{figure}[H]
    \centering
    \includegraphics[width=0.9\linewidth]{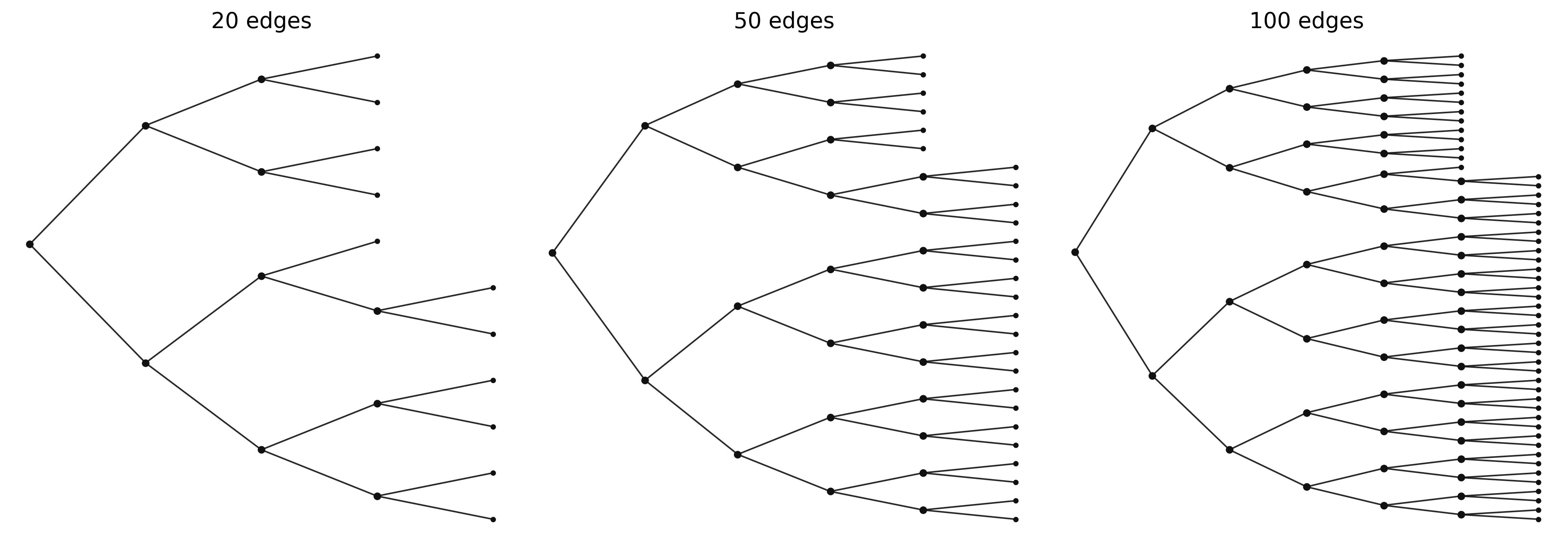}
    \caption{Nearly balanced binary metric trees used in the network-scaling experiment. Thin lines denote edges and points denote vertices.}
    \label{fig:large_networks}
\end{figure}

The relative mean and RMS errors in \Cref{tab:large_networks} are normalized by the exact $L^2(\mathcal G)$ norm at each time and then maximized over the temporal grid. The speedup and memory-reduction columns are the Full-Graph/RBM ratios. The measured speedup increases from $2.00$ to $5.42$, while the storage reduction increases from $1.49$ to $4.93$ as the graph grows from $20$ to $100$ edges.

\begin{table}[H]
\centering
\caption{Accuracy, execution time, and retained solver storage on the family of large metric trees. RBM values correspond to the connected overlapping decomposition.}
\label{tab:large_networks}
\resizebox{\textwidth}{!}{
\begin{tabular}{ccclcccccc}
\toprule
Edges & Batches & DOFs & Method & Relative mean error & Relative RMS error & Time (s) & Storage (MB) & Speedup & Storage reduction \\
\midrule
20 & -- & 1990 & Full-Graph FEM & $2.9388\times10^{-4}$ & $2.9388\times10^{-4}$ & $0.40621\pm0.01234$ & $0.144$ & $1.00$ & $1.00$ \\
20 & 5 & 1990 & Discretize+RBM & $7.2127\times10^{-3}$ & $3.7104\times10^{-2}$ & $0.20265\pm0.01489$ & $0.097$ & $2.00$ & $1.49$ \\
50 & -- & 4975 & Full-Graph FEM & $3.0683\times10^{-4}$ & $3.0683\times10^{-4}$ & $1.03840\pm0.03088$ & $0.740$ & $1.00$ & $1.00$ \\
50 & 10 & 4975 & Discretize+RBM & $1.3936\times10^{-2}$ & $6.1575\times10^{-2}$ & $0.27336\pm0.01287$ & $0.270$ & $3.80$ & $2.74$ \\
100 & -- & 9950 & Full-Graph FEM & $3.1016\times10^{-4}$ & $3.1016\times10^{-4}$ & $1.84194\pm0.02304$ & $2.695$ & $1.00$ & $1.00$ \\
100 & 20 & 9950 & Discretize+RBM & $2.3480\times10^{-2}$ & $9.5535\times10^{-2}$ & $0.34013\pm0.03531$ & $0.547$ & $5.42$ & $4.93$ \\
\bottomrule
\end{tabular}}
\end{table}

    \subsubsection{Optimal Control on a Graph}
    
    Let us consider the same graph $\mathcal{G}$ introduced in the previous section. Denote by $(y,f)$ the optimal pair of the optimal control problem \eqref{eq:optimal_control} restricted to \eqref{eq:heat_net}-\eqref{eq:heat_bc} (replacing the norms in $L^2(0,L)$ with norms in $L^2(\mathcal{E})$). In this section, we will use the Discretize+RBM methodology to approximate the optimal pair. Let us consider the discrete optimal control \eqref{eq:finite_dim_ffunctional} restricted to \eqref{eq:variational_relation_y_graph}, and denote its optimal pair as $(y_h,f_h)$. Due to \cite{MR1119274}, \Cref{th:conver_optimal_control} holds for the optimal pairs $(y,f)$ and $(y_h,f_h)$ (replacing $L^2(0,L)$ with $L^2(\mathcal{E})$). Similarly to \eqref{eq:characterization_norm_l2}, we observe that the vector of coefficient of $(y_h,f_h)$, satisfies the problem
    \begin{align}\label{eq:funtional_D2}
       \min_{{\F_h} \in  L^2(0,T;\mathbb{R}^{\mathcal{N}})} \left\{ J_{\boldsymbol{E}_h}({\F_h}) = \frac{1}{2} \int_0^T \left( \| {\F_h}(t)\|^2_{\boldsymbol{E}_h} + \|\Y_h(t) - \y_{d,h}\|^2_{\boldsymbol{E}_h} \right) dt \right\},
    \end{align}
    with $\mathcal{N}=10N+3$ and $\Y_h$ being the solution of \eqref{eq:discrete_network_fem}, $\boldsymbol{E}_h$ the {mass} matrix introduced in \Cref{subsec:numerics_equation}, and $\y_{d,h}$ the projection of $y_d$ onto the space $V^{\mathcal{E}}_h$.  Finally, consider the $\bomega$-dependent optimal control problem:
    \begin{align}\label{eq:funtional_D2_random}
        \min_{{\F_R} \in L^2(0,T;\mathbb{R}^{\mathcal{N}})} \left\{ J^R_{\boldsymbol{E}_h}({\F_R}) = \frac{1}{2} \int_0^T \left( \| {\F_R}(t)\|^2_{\boldsymbol{E}_h} + \|\Y_R(t) - \y_{d,h}\|^2_{\boldsymbol{E}_h} \right) dt \right\},
    \end{align}
For the numerical experiment, the consistent mass matrix $\boldsymbol E_h$ in \eqref{eq:funtional_D2} and \eqref{eq:funtional_D2_random} is replaced consistently by $\boldsymbol E_h^\ell$ in the objective functional, state equation, adjoint equation, and gradient norm. For each realization, the RBM state satisfies $\boldsymbol E_h^\ell\partial_t\Y_R+\boldsymbol R_{rb}\Y_R=\boldsymbol E_h^\ell\F_R,\qquad \Y_R(0)=0.$ For every realization of the batch process, we solve the corresponding pathwise optimal control problem. The batch sequence is fixed during the optimization, the state equation follows this sequence forward in time, and the discrete adjoint uses the same sequence in reverse order. The reported RBM state and control are the empirical means of the resulting realization-dependent optimal states and controls.

For both problems, we use $N=30$ private interior nodes per edge, giving $303$ global degrees of freedom, and $300$ temporal collocation points. Thus $\Delta t=\delta=1/299$. The initial state and initial control are zero, and the desired state is $y_d=1$. The optimization is initialized with the zero control and solved by gradient descent with Armijo backtracking, initial step length one, contraction factor $1/2$, and Armijo constant $10^{-4}$. The stopping condition is
\[
\frac{\|\nabla J(\F^k)\|_{L^2_{\Delta t}(\boldsymbol{E}_h^\ell)}}{\max\{1,\|\F^k\|_{L^2_{\Delta t}(\boldsymbol{E}_h^\ell)}\}}\leq10^{-8}.
\]
The RBM statistics are obtained from $20$ independent pathwise optimal control problems.

\begin{figure}[H]
    \centering
    \includegraphics[width=0.9\linewidth]{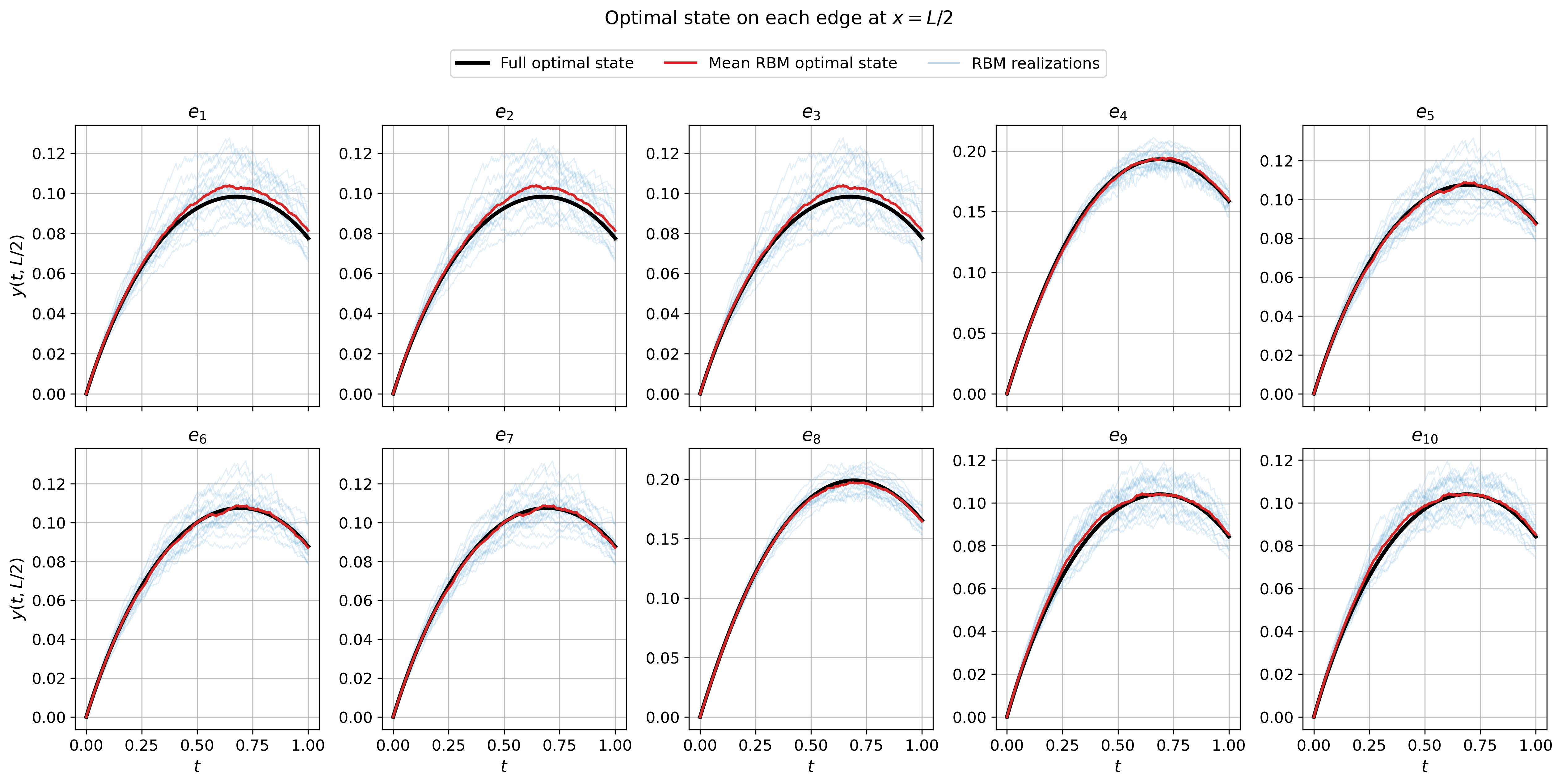}
    \caption{Full-Graph optimal state, empirical mean of the RBM optimal states, and the individual pathwise optimal states at $x=L/2$ on every edge.}
    \label{fig:optimal_state}
\end{figure}

\begin{figure}[H]
    \centering
    \includegraphics[width=0.9\linewidth]{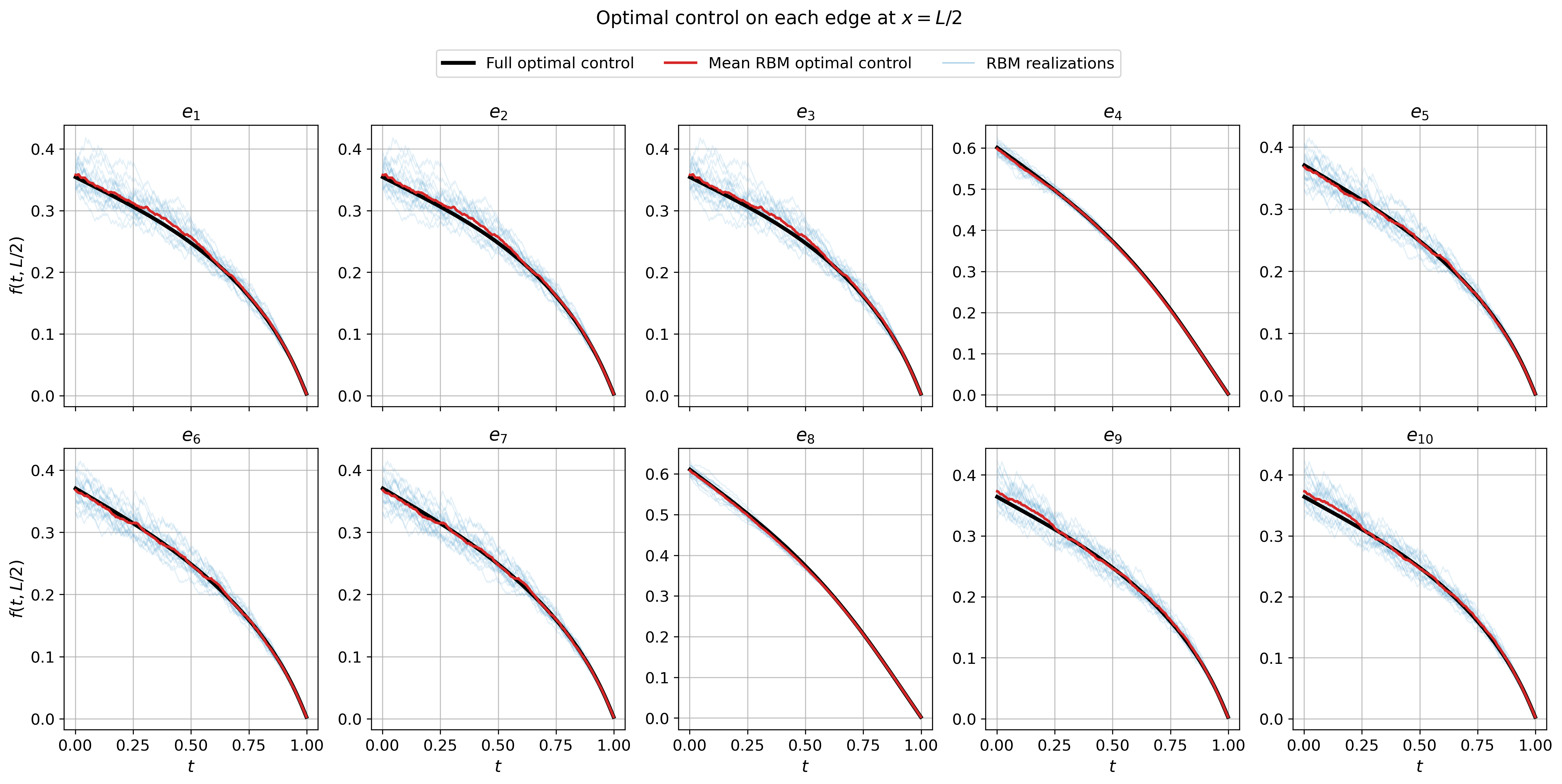}
    \caption{Full-Graph optimal control, empirical mean of the RBM optimal controls, and the individual pathwise optimal controls at $x=L/2$ on every edge.}
    \label{fig:optimal_control}
\end{figure}

The relative $L^2(0,T;L^2(\mathcal G))$ difference between the Full-Graph optimal state and the empirical RBM optimal state is $1.85\%$. The corresponding difference between the optimal controls is $1.53\%$. The Full-Graph optimization converges in $10$ iterations, while the RBM realizations require $10.15\pm0.37$ iterations.

Table~\ref{tab:memory_oc} reports both the measured execution time and a hardware-independent measure of the linear-solve work. One work unit represents one retained band-factor entry used in each of the two triangular substitutions, accumulated over all state and adjoint solves performed by the optimization. The RBM requires $34.2\%$ of the Full-Graph linear-solve work, corresponding to a reduction of $65.8\%$. At $303$ degrees of freedom, the measured times remain close because the vector operations, indexing, and function-call costs are comparable to the small linear solves.

\begin{table}[H]
\centering
\caption{Gradient-descent convergence and computational cost for the Full-Graph and pathwise RBM optimal control problems. RBM entries are means and sample standard deviations over $20$ realizations.}
\label{tab:memory_oc}
\resizebox{\textwidth}{!}{
\begin{tabular}{lcccccc}
\toprule
Method & Iterations & Relative gradient & Execution time (s) & Factor storage (MB) & Linear-solve work & Relative work \\
\midrule
Full-Graph FEM & $10$ & $3.665\times10^{-9}$ & $0.3208\pm0.0229$ & $0.0150$ & $1.9931\times10^7$ & $1.000$ \\
Discretize+RBM & $10.15\pm0.37$ & $(3.955\pm1.240)\times10^{-9}$ & $0.3425\pm0.0234$ & $0.0132$ & $6.8182\times10^6$ & $0.342$ \\
\bottomrule
\end{tabular}}
\end{table}

    \section{Conclusions and Open Problems}\label{sec:c_and_o}
    {
    In this work we have introduced and analyzed a discretize+RBM strategy for the heat equation on one-dimensional graphs. The equation is first discretized in space by the finite element method, and then the stiffness matrix of the semi-discrete system is decomposed into a small number of blocks that are activated in a randomized way on short time intervals, while the mass matrix and the spatial discretization remain fixed.}
    
    {For the heat equation on an interval we derived mean-square error bounds between the random-batch solution and the exact PDE solution. These bounds exhibit a coupled dependence on the mesh size and on the RBM switching parameter and show that, under a suitable joint refinement, the random-batch solution converges in expectation to the continuous heat equation in the natural space–time norms. The same framework was then applied to a linear-quadratic optimal control problem, yielding error estimates for both the optimal state and the optimal control when the dynamics are replaced by their random-batch counterparts.}
    
    {The methodology extends naturally to heat equations on metric graphs. After finite element discretization on each edge and incorporation of the vertex conditions, the RBM acts at the matrix level and can be viewed as a randomized variant of domain decomposition on graphs, evolving selected edge groups at each step. The numerical experiments confirm convergence under the coupled refinement of $h$ and $\delta$, validate the method for solutions with nonzero interior-vertex values, and demonstrate increasing computational savings as the network grows. On the family of binary metric trees, the measured speedup increases from $2.00$ to $5.42$ and the solver-storage reduction from $1.49$ to $4.93$. In the optimal control experiment, the empirical RBM optimal state and control differ from their Full-Graph counterparts by $1.85\%$ and $1.53\%$, respectively, while the accumulated linear-solve work is reduced by $65.8\%$.}
    
    {These features make the proposed discretize+RBM framework a useful tool for PDE models on networks, such as thermal or gas-flow processes in pipeline and energy infrastructures, where graph domains arise naturally and computational efficiency is a central concern.
    }
    \medbreak
    
    Despite these findings, several open questions remain.
    \smallbreak
    {\bf 1) Extension of the Finite-Dimensional RBM:} Although \cite{MR4433122} provides a comprehensive analysis of RBM in finite-dimensional systems, there is currently no general theory for nonlinear ordinary differential equations (ODEs). RBM for nonlinear systems could be applied to power system models used in engineering that also involve graphs, such as Power Systems (see \cite[Section 3.5]{machowski2020power}). Additionally, like the present study, RBM for nonlinear systems could be extended to nonlinear PDEs.
    
    On the other hand, a natural question that arises when introducing \eqref{eq:random_system} is whether it is possible to not only randomize the stiffness matrix $R_h$ but also to randomize the mass matrix $E_h$. This is not straightforward since \Cref{remark:mass_matrix} no longer holds, and the results in \cite{MR4433122} cannot be directly applied. Therefore, a careful analysis is required to prove (or refute) the convergence when $E_h$ is randomized.

    {\bf 2) Infinite-dimensional RBM:} As shown in this article, it is possible to use the RBM techniques developed in \cite{MR4433122} for ODEs and apply them in the context of PDEs, after spatial discretization. 
    
    It would be of interest to develop an RBM at the operator level, as in \cite{eisenmann2022randomized}, but without relying on time discretization. Specifically, one could aim for a space-time continuous RBM for PDEs based on a randomized splitting of the operator generating the dynamics. Such an approach may yield convergence guarantees that are independent of the spatial discretization parameter $h$, depending only on the RBM switching parameter $\delta$. Note that a true space-time continuous formulation cannot be inferred from \cite{eisenmann2022randomized}, as their analysis ties the time discretization, the implicit Euler scheme, and the RBM switching parameter $\delta$ together.

    Conversely, in the context of domain decomposition (see \cite[Chapter III]{MR2093789} and \cite{MR1857663}), the transmission conditions prescribed at subdomain interfaces are both critical and scheme-dependent. In non-overlapping methods, one typically enforces Dirichlet–Neumann or Robin–Robin coupling to ensure stability and convergence, whereas overlapping Schwarz‐type algorithms exploit the overlap width to dampen interfacial errors, imposing fewer constraints on boundary exchanges. In contrast, our approach omits explicit interface conditions because the RBM approximates the PDE through a discretized PDE, requiring the coupled refinement $\delta=o(h^6)$ to ensure convergence. To avoid such dependence, it is necessary to directly approximate the PDE with a fully continuous RBM scheme. To this end, it is crucial to determine the precise interface laws that ensure the scheme's well-posedness and convergence on average.

    {\bf 3) Graph Partitioning:} In our numerical examples, we have demonstrated a partition inspired by overlapping and non-overlapping domain methods. However, since \Cref{th:convergence_RBM_FE} holds for any splitting of the stiffness matrix $R_h$, the question naturally arises of identifying a graph partition that optimizes the approximation. Three strategies for graph decomposition that may be of interest are as follows. The first is based on \textit{algebraic connectivity or Fiedler value}, namely, the second eigenvalue of the graph Laplacian, which quantifies the graph's connectivity and enables a partition based on this measure. The second approach employs \textit{machine learning} techniques, particularly the \textit{Lloyd algorithm}—an iterative method (based on Voronoi diagrams) used to divide the domain into regions. Although this method may serve as an initial step toward a partition that accelerates the RBM, it requires the incorporation of dynamic information. One possibility is to consider a weighted graph (with weights on the edges or vertices) that captures partial information about the solution at each time, and then apply the aforementioned techniques. Finally, as indicated in \Cref{remark:conv_ht_hx}, accelerating convergence requires minimizing the constant $C(M)$, which depends on the chosen splitting $\{R^h_m\}_{m=1}^M$. Therefore, for a fixed $M$, one may formulate an \textit{optimization problem} to simultaneously minimize $C(M)$ and the $l^1$ (or $l^0$) norm of the matrices $\{R_m\}_{m=1}^M$, thereby promoting sparsity. We refer to \cite{alvarezhernandez} for a deeper discussion of batch selection and the trade-off between memory usage and accuracy in the context of nonlinear ODEs.

    \section*{Acknowledgments}
    The authors wish to express their gratitude to F. Hante and J. Giesselmann for their insightful discussions, and to S. Zamorano for taking the time to critically review the manuscript. 
    
    M. Hern\'{a}ndez has been funded by the Transregio 154 Project, Mathematical Modelling, Simulation, and Optimization Using the Example of Gas Networks of the DFG, project C07, the fellowship ``ANID-DAAD bilateral agreement", and NSF grant No. 2529107. E. Zuazua was partially supported by the European Research Council (ERC) under the European Union's Horizon Europe research and innovation programme (grant agreement No.~101096251-CoDeFeL); by the Alexander von Humboldt Professorship program; the European Union's Horizon Europe MSCA project ModConFlex (HORIZON-MSCA-2021-DN-01, project 101073558); the Transregio 154 Project ``Mathematical Modelling, Simulation and Optimization Using the Example of Gas Networks'' of the DFG; the AFOSR 24IOE027 project; the SURE-AI Norwegian Centre for Sustainable, Risk-Averse, and Ethical AI grant 357482, Research Council of Norway; by the Grant PID2023-146872OB-I00-DyCMaMod of MICIU (Spain), by the COST Actions CA24122 -- Multiscale Stochastics, Patterns, and Analysis of Combinatorial Environments and CA24136 -- Interactions between Control Theory and Machine Learning and by the Basque Government through the consolidated research group ”CMRG-Computational Mathematics Research Group” under Grant IT1856-26.


\begin{thebibliography}{00}
    
    \bibitem{MR4433122}
    D.~W.~M.~Veldman and E.~Zuazua,
    \emph{A framework for randomized time-splitting in linear-quadratic optimal control},
    Numer.\ Math., \textbf{151} (2022), no.~2, 495--549.
     {10.1007/s00211-022-01290-3}
    
    \bibitem{eisenmann2022randomized}
    M.~Eisenmann and T.~Stillfjord,
    \emph{A randomized operator splitting scheme inspired by stochastic optimization methods},
    Numer.\ Math., \textbf{156} (2024), no.~2, 435--461.
     {10.1007/s00211-024-01396-w}
    
    \bibitem{MR3058594}
    S.~Ervedoza and E.~Zuazua,
    \emph{Numerical approximation of exact controls for waves},
    SpringerBriefs in Mathematics, Springer, New York, 2013.
     {10.1007/978-1-4614-5808-1}
    
    \bibitem{MR1036928}
    R.~Glowinski, W.~Kinton, and M.~F.~Wheeler,
    \emph{A mixed finite element formulation for the boundary controllability of the wave equation},
    Internat.\ J.\ Numer.\ Methods Engrg., \textbf{27} (1989), no.~3, 623--635.
     {10.1002/nme.1620270313}
    
    \bibitem{MR1065445}
    R.~Glowinski and C.~H.~Li,
    \emph{On the numerical implementation of the Hilbert uniqueness method for the exact boundary controllability of the wave equation},
    C.~R.\ Acad.\ Sci.\ Paris S\'er.\ I Math., \textbf{311} (1990), no.~2, 135--142.
    
    \bibitem{MR1039237}
    R.~Glowinski, C.~H.~Li, and J.-L.~Lions,
    \emph{A numerical approach to the exact boundary controllability of the wave equation. I. Dirichlet controls: description of the numerical methods},
    Japan J.\ Appl.\ Math., \textbf{7} (1990), no.~1, 1--76.
     {10.1007/BF03167891}
    
    \bibitem{MR4230431}
    S.~Jin, L.~Li, and J.-G.~Liu,
    \emph{Convergence of the random batch method for interacting particles with disparate species and weights},
    SIAM J.\ Numer.\ Anal., \textbf{59} (2021), no.~2, 746--768.
     {10.1137/20M1327641}
    
    \bibitem{MR4436794}
    J.~A.~Carrillo, S.~Jin, and Y.~Tang,
    \emph{Random batch particle methods for the homogeneous Landau equation},
    Commun.\ Comput.\ Phys., \textbf{31} (2022), no.~4, 997--1019.
     {10.4208/cicp.oa-2021-0200}
    
    \bibitem{MR4553236}
    S.~Jin, L.~Li, X.~Ye, and Z.~Zhou,
    \emph{Ergodicity and long-time behavior of the Random Batch Method for interacting particle systems},
    Math.\ Models Methods Appl.\ Sci., \textbf{33} (2023), no.~1, 67--102.
     {10.1142/S0218202523500033}
    
    \bibitem{MR4307003}
    D.~Ko and E.~Zuazua,
    \emph{Model predictive control with random batch methods for a guiding problem},
    Math.\ Models Methods Appl.\ Sci., \textbf{31} (2021), no.~8, 1569--1592.
     {10.1142/S0218202521500329}
    
    \bibitem{MR3797719}
    L.~Bottou, F.~E.~Curtis, and J.~Nocedal,
    \emph{Optimization methods for large-scale machine learning},
    SIAM Rev., \textbf{60} (2018), no.~2, 223--311.
     {10.1137/16M1080173}
    
    \bibitem{MR1829182}
    D.~Q.~Mayne, J.~B.~Rawlings, C.~V.~Rao, and P.~O.~M.~Scokaert,
    \emph{Constrained model predictive control: stability and optimality},
    Automatica J.\ IFAC, \textbf{36} (2000), no.~6, 789--814.
     {10.1016/S0005-1098(99)00214-9}
    
    \bibitem{MR3585330}
    L.~Gr\"une and J.~Pannek,
    \emph{Nonlinear model predictive control}, 2nd ed.,
    Communications and Control Engineering Series, Springer, Cham, 2017.
     {10.1007/978-3-319-46024-6}
    
    \bibitem{MR1625845}
    L.~C.~Evans,
    \emph{Partial differential equations}, 2nd ed.,
    Graduate Studies in Mathematics, vol.~19, American Mathematical Society, Providence, RI, 2010.
     {10.1090/gsm/019}
    
    \bibitem{MR2583281}
    F.~Tr\"{o}ltzsch,
    \emph{Optimal control of partial differential equations},
    Graduate Studies in Mathematics, vol.~112, American Mathematical Society, Providence, RI, 2010.
     {10.1090/gsm/112}
    
    \bibitem{MR1299729}
    A.~Quarteroni and A.~Valli,
    \emph{Numerical approximation of partial differential equations},
    Springer Series in Computational Mathematics, vol.~23, Springer-Verlag, Berlin, 1994.
    
    \bibitem{Egger}
    H.~Egger and N.~Philippi,
    \emph{On the transport limit of singularly perturbed convection-diffusion problems on networks},
    Math.\ Methods Appl.\ Sci., \textbf{44} (2021), no.~6, 5005--5020.
     {10.1002/mma.7084}
    
    \bibitem{MR3243602}
    D.~Mugnolo,
    \emph{Semigroup methods for evolution equations on networks},
    Understanding Complex Systems, Springer, Cham, 2014.
     {10.1007/978-3-319-04621-1}
    
    \bibitem{MR1119274}
    M.~Kr\"oller and K.~Kunisch,
    \emph{Convergence rates for the feedback operators arising in the linear quadratic regulator problem governed by parabolic equations},
    SIAM J.\ Numer.\ Anal., \textbf{28} (1991), no.~5, 1350--1385.
     {10.1137/0728071}
    
    \bibitem{MR2478556}
    A.~Iserles,
    \emph{A first course in the numerical analysis of differential equations}, 2nd ed.,
    Cambridge Texts in Applied Mathematics, Cambridge University Press, Cambridge, 2009.
    
    \bibitem{veldman2023stability}
    D.~W.~M.~Veldman, A.~Borkowski, and E.~Zuazua,
    \emph{Stability and Convergence of a Randomized Model Predictive Control Strategy},
    IEEE Trans.\ Autom.\ Control (2024), 1–8.
     {10.1109/TAC.2024.3375253}
    
    \bibitem{MR2373954}
    S.~C.~Brenner and L.~R.~Scott,
    \emph{The mathematical theory of finite element methods}, 3rd ed.,
    Springer, New York, 2008.
     {10.1007/978-0-387-75934-0}
    
    \bibitem{corella2024minibatchdescentsemiflows}
    A.~Domínguez Corella and M.~Hernández,
    \emph{Mini-batch descent in semiflows},
    ESAIM Control Optim.\ Calc.\ Var.\ (2025).
     {10.1051/cocv/2025018}


\bibitem{alvarezhernandez}
A.~Álvarez-López and M.~Hernández,
\emph{Convergence, design and training of continuous-time dropout as a random batch method},
arXiv:2510.13134 (2026).
    
    \bibitem{ShietalRBM}
    S.~Jin, L.~Li, and J.-G.~Liu,
    \emph{Random batch methods (RBM) for interacting particle systems},
    J.\ Comput.\ Phys., \textbf{400} (2020), 108877, 30 pp.
     {10.1016/j.jcp.2019.108877}
    
    \bibitem{Latz}
    J.~Latz,
    \emph{Analysis of stochastic gradient descent in continuous time},
    Stat.\ Comput., \textbf{31} (2021), no.~4, Paper No.~39, 25 pp.
     {10.1007/s11222-021-10016-8}
    
    \bibitem{tridiagonal}
    D.~Kulkarni, D.~Schmidt, and S.-K.~Tsui,
    \emph{Eigenvalues of tridiagonal pseudo-Toeplitz matrices},
    Linear Algebra Appl., \textbf{297} (1999), no.~1–3, 63--80.
     {10.1016/S0024-3795(99)00114-7}
    
    \bibitem{MR4361973}
    S.~Jin and L.~Li,
    \emph{On the mean field limit of the random batch method for interacting particle systems},
    Sci.\ China Math., \textbf{65} (2022), no.~1, 169--202.
     {10.1007/s11425-020-1810-6}
    
    \bibitem{MR4300142}
    D.~Ko, S.-Y.~Ha, S.~Jin, and D.~Kim,
    \emph{Uniform error estimates for the random batch method to the first-order consensus models with antisymmetric interaction kernels},
    Stud.\ Appl.\ Math., \textbf{146} (2021), no.~4, 983--1022.
     {10.1111/sapm.12372}
    
    \bibitem{MR4744252}
    L.~Pareschi and M.~Zanella,
    \emph{Reduced variance random batch methods for nonlocal PDEs},
    Acta Appl.\ Math., \textbf{191} (2024), Paper No.~4, 23 pp.
     {10.1007/s10440-024-00656-z}
    
    \bibitem{MR4175147}
    Y.~Qiu, S.~Grundel, M.~Stoll, and P.~Benner,
    \emph{Efficient numerical methods for gas network modeling and simulation},
    Netw.\ Heterog.\ Media, \textbf{15} (2020), no.~4, 653--679.
     {10.3934/nhm.2020018}
    
    \bibitem{MR3815942}
    Y.~Shi et al. (eds.),
    \emph{Computational science—ICCS 2018. Part III},
    Lecture Notes in Computer Science, vol.~10862, Springer, Cham, 2018.
    
    \bibitem{machowski2020power}
    J.~Machowski, Z.~Lubosny, J.~W.~Bialek, and J.~R.~Bumby,
    \emph{Power system dynamics: stability and control},
    John Wiley \& Sons, 2020.
    
    \bibitem{MR3931345}
    C.~Beisbart and N.~J.~Saam (eds.),
    \emph{Computer simulation validation},
    Simulation Foundations, Methods and Applications, Springer, Cham, 2019.
     {10.1007/978-3-319-70766-2}
    
    \bibitem{MR2093789}
    J.~E.~Lagnese and G.~Leugering,
    \emph{Domain decomposition methods in optimal control of partial differential equations},
    Int.\ Ser.\ Numer.\ Math., vol.~148, Birkhäuser, Basel, 2004.
     {10.1007/978-3-0348-7885-2}
    
    \bibitem{MR1857663}
    A.~Quarteroni and A.~Valli,
    \emph{Domain decomposition methods for partial differential equations},
    Numer.\ Math.\ Sci.\ Comput., Oxford Science Publications, Clarendon Press, Oxford Univ.\ Press, New York, 1999.
    
    \bibitem{MR4879947}
    M.~Badra and J.-P.~Raymond,
    \emph{Approximation of feedback gains for abstract parabolic systems},
    ESAIM Control Optim.\ Calc.\ Var., \textbf{31} (2025), Paper No.~17, 39 pp.
     {10.1051/cocv/2024084}
    
     \bibitem{DAUS2022111220}
    E.~S.~Daus, M.~Fellner, and A.~Jüngel,
    \emph{Random‐batch method for multi‐species stochastic interacting particle systems},
    J.~Comput.~Phys., \textbf{463} (2022), 111220.
    DOI: 10.1016/j.jcp.2022.111220
    
    \bibitem{MR3728372}
    F.~M.~Hante, G.~Leugering, A.~Martin, L.~Schewe and M.~Schmidt.
    Challenges in optimal control problems for gas and fluid flow in networks of pipes and canals: from modeling to industrial application.
    En: \emph{Industrial Mathematics and Complex Systems}, Ind. Appl. Math., pp.~77--122. Springer, Singapore, 2017.
    
    \bibitem{Leuge}
    G.~Leugering.
    Partial Differential Equations on Metric Graphs: A Survey of Results on Optimization, Control, and Stabilizability Problems with Special Focus on Shape and Topological Sensitivity Problems.
    En: \emph{Mathematical Modelling, Optimization, Analytic and Numerical Solutions}, pp.~77--115.
    Springer Singapore, Singapore, 2020.
    \textsc{doi}: 10.1007/978-981-15-0928-5\_4.


    \bibitem{FujitaMizutani1976}
H.~Fujita and A.~Mizutani,
\emph{On the finite element method for parabolic equations, I: Approximation of holomorphic semi-groups},
J. Math. Soc. Japan, \textbf{28} (1976), no.~4, 749--771.
\textsc{doi}: 10.2969/jmsj/02840749.

\bibitem{Thomee2006}
V.~Thom\'ee,
\emph{Galerkin Finite Element Methods for Parabolic Problems},
Springer Series in Computational Mathematics, vol.~25, 2nd ed., Springer, Berlin, 2006.
\textsc{doi}: 10.1007/3-540-33122-0.
    
    \end{thebibliography}
\end{document}